\documentclass[11pt,leqno]{article}

\usepackage{geometry}

\usepackage{graphicx}
\usepackage{epstopdf,epsfig}

\usepackage[font=small]{caption}
\usepackage{color}

\usepackage{amsmath,amsthm,latexsym,amscd,esint,enumerate}
\usepackage{amsfonts,amssymb}
\usepackage{MnSymbol}

\usepackage[scr=boondoxo,scrscaled=1.05]{mathalfa}

\usepackage{xy}
\xyoption{all}

\usepackage[explicit]{titlesec}

\usepackage{amssymb}
\usepackage{amscd}
\usepackage{hyperref}
\usepackage{enumitem}
\usepackage{tikz}
\usepackage{pgf}

\swapnumbers

\theoremstyle{plain}
    \newtheorem{theorem}{Theorem}[section]
    \newtheorem{lemma}[theorem]{Lemma}
    \newtheorem*{lemma*}{Lemma}
    \newtheorem*{proposition*}{Proposition}
        \newtheorem*{corollary*}{Corollary}
    \newtheorem*{theorem*}{Theorem}
    \newtheorem{corollary}[theorem]{Corollary}
    \newtheorem{fact}[theorem]{Fact}
    \newtheorem{proposition}[theorem]{Proposition}

\theoremstyle{definition}
    \newtheorem{definition}[theorem]{Definition}

    \newtheorem{notation}[theorem]{Notation}
    
    \newtheorem{remark}[theorem]{Remark}

\newtheorem*{ack}{Acknowledgments}

\theoremstyle{remark}

\numberwithin{equation}{section}

\newcounter{actr}

\newcounter{bctr}

\newcommand{\Hawaii}{Hawai\kern.05em`\kern.05em\relax i}

\newcommand{\R}{\mathbb{R}}

\newcommand{\Z}{\mathbb{Z}} 
\newcommand{\Q}{\mathbb{Q}}

\newcommand{\E}{\mathbb{E}}

\newcommand{\K}{\mathbb{K}}

\newcommand{\Ad}{\mathrm{Ad}}

\DeclareMathOperator{\GL}{GL}

\DeclareMathOperator{\Isom}{Isom}

\begin{document}   

\title{Concentration inequalities for maximal displacement of random walks on groups of polynomial growth}
\date{}

\author{\scshape 
J. Brieussel%
\footnote{J.~B.~was supported in part by ANR-24-CE40-3137 and JSPS IF-L25508}, 
R.~Tessera%
\footnote{R.~T.~was supported in part by ANR-24-CE40-3137},%
T.~Zheng%
\footnote{T.~Z.~was supported in part by NSF DMS-2538551}%
}
\AtEndDocument{\bigskip{\footnotesize%
 \par
\addvspace{\medskipamount}
}}

\maketitle

\begin{abstract}
We prove concentration inequalities for maximal displacement of
compactly supported random walks on a compactly generated locally
compact group with polynomial growth. The probability that a centred random
walk escapes the ball of radius $t\sqrt{n}$ before time $n$ decays
like $\exp(O(-t^{2}))$. Similar bounds with different exponents hold for non-centred random walks as well, after correction by the drift. 

When the support of the measure generates a virtually nilpotent group, we provide an effective version of this result. These more refined estimates rely on the existence of a ``quantitative splitting'' of a virtually simply connected nilpotent group, a result which may be of independent interest. 

As applications, we deduce that the same
concentration inequalities hold for centred random walks on the following classes of groups: amenable connected
Lie groups (including non-unimodular ones), polycyclic and more generally
finitely generated solvable groups with finite Pr\"ufer rank. In
particular the maximal drift (in $L^{p}$ for all $p\geq1$) of such
a random walk at time $n$ is in $O(n^{1/2})$. For polycyclic groups, this strengthens and completes partial results previously obtained by Russ Thompson in
2011. 
\end{abstract}

\section{Introduction}

A concentration inequality is an upper bound on the tail of a probability measure away from its mean. While central limit style theorems provide accurate descriptions of asymptotic distributions of random walks at the scale of their variance, concentration results provide less precise but effective bounds on the tail of the distribution at any scale of space and any time. 

A family of real random variables $(f_i)_{i\in I}$ is called to have sub-stretched-exponential concentration with exponent $\alpha$, $\alpha$-concentrated for short, for some $\alpha>0$ if there exists constants $c_1,c_2 >0$ such that
\[
\forall t>0, \quad \sup_{i\in I} \mathbb{P}\left(|f_i|\ge t\right) \le c_2\exp\left(-c_1 t^{\alpha}\right).
\]
For $\alpha=1$ or $\alpha=2$, the family is called subexponential or subgaussian respectively. 

In this paper, we shall be concerned with random walks on compactly generated locally compact groups of polynomial growth, driven by a compactly supported measure $\mu$. Our goal is to show a concentration inequality for the distance to the origin after $n$ steps of the random walk. An important aspect of our result is that we do not assume that $\mu$ is absolutely continuous with respect to the Haar measure, nor that it is symmetric, or even centred. In the case where it is centred, we actually obtain that this distance, renormalized by $n^{1/2}$ is subgaussian. Before stating our precise result, let us briefly recall the history of the subject. 

The study of the long-time behavior of random walks on groups is closely associated with Varopoulos, who introduced powerful analytic tools (those are summarized in the book \cite{Varobook}). These methods were subsequently refined and extended by many authors, including Alexopoulos, Coulhon, Hebisch, and Saloff-Coste. Generally, this study goes in two steps: first one provides estimates on the probability of return to the origin (these are generally referred to as ``on-diagonal'' estimates). In the case of a group $G$ of polynomial growth, equipped with a non-degenerate, compactly supported, symmetric probability $\mu$ which is absolutely continuous with respect to the Haar measure, these estimates are of the form 
\begin{equation}\label{eq:Ondiag}
\phi_n(e_G)\leq C n^{-d/2},
\end{equation}
where $C\geq 1$ is a constant, $d$ is the degree of growth of $G$, and $\phi_n$ is the density of $\mu^{*n}$ with respect to the Haar measure. This upper bound is essentially due to Varopoulos (see \cite[Theorem 4.1]{HebishSaloff} for the precise statement).

The second step consists in deducing ``off-diagonal'' pointwise Gaussian upper estimates of the form 
\begin{equation}\label{eq:upperGaussianOffdiag}
\phi_n(g)\leq Cn^{-d/2}\exp\left(-|g|^2/(Cn)\right),
\end{equation}
where $C\geq 1$ is a constant, and $|g|$ denotes the word length of $g$ with respect to some compact generating subset. Indeed, in their influential article \cite{HebishSaloff}, Hebisch and Saloff-Coste show that the upper bound (\ref{eq:upperGaussianOffdiag})  can be derived via an interpolation argument, using the on-diagonal estimate (\ref{eq:Ondiag}) as an input. 

Note that (\ref{eq:upperGaussianOffdiag}) should not be mistaken with the very general Carne-Varopoulos Gaussian upper bound \cite{Varo85,Carne}, which is valid for any reversible Markov chain, where the metric $d(x,y)$ is taken to be the minimal number of steps  required for the Markov chain to go from $x$ to $y$. It yields in our case a Gaussian upper bound of the form
\[\phi_n(g)\leq C\exp\left(-|g|^2/(Cn)\right),\]
namely, {\it without} the corrective factor $n^{-d/2}$. Indeed, the importance of this term is illustrated by the following remark: summing up (\ref{eq:upperGaussianOffdiag}) over the complement of the ball of radius $tn^{1/2}$, the factor $n^{-d/2}$ compensates the volume growth, allowing us to deduce the following tale estimate
\begin{equation}\label{eq:concentration}
\sup_{g\in G}\mu^{*n}(g\mid |g|\geq tn^{1/2})\leq C'\exp\left(-t^2/C'\right).
\end{equation}

In words, this inequality says that the probability that the random walks at time $n$ lands outside the ball of radius $tn^{1/2}$ is subgaussian. Yet this is not as strong as the concentration inequality we are aiming for, which claims that the probability that the random walks exits the ball of radius $tn^{1/2}$ {\it before time $n$} is subgaussian. However the latter can be deduced as well from (\ref{eq:upperGaussianOffdiag}) \cite[Lemma 13.3]{Alexopoulos}.

The assumption that the measure $\mu$ is symmetric plays an important role in the techniques for obtaining heat kernel estimates, which exploit the self-adjointness of the corresponding Markov operator. But more crucial is the assumption that the measure is absolutely continuous with respect to the Haar measure. Indeed,  Alexopoulos managed to push these techniques one step further and prove (\ref{eq:upperGaussianOffdiag}) for centred random walks \cite{Alexopoulos, AlexopoulosLC}, see also Dungey~\cite{Dungey} for finer results.

Given a set $A$, denote by $\tau_A=\inf\{n: w_n\notin A\}$ the first exit time from $A$. Estimates on the first exit time of balls play an important role in the study of random walks on graphs, and also diffusion semigroups. We mention that there are general arguments showing that a two-sided estimate on the mean exit time $\mathbb{E}_x[\tau_{B(x,r)}]$ implies a tail estimate on $\mathbb{P}_x(\tau_{B(x,r)}\le t)$, see for example, \cite[Theorem 3.15]{GT}.  Results of this type go back to Barlow \cite{Barlow95}. The relations between exit time estimates, transition probabilities $p^{n}(x,y)$, and various functional inequalities have been object of intensive research in the past decades, we refer the readers to \cite{GT} and references therein. 

In our setting, we consider general random walks driven by a measure with compact support, where known tools in heat kernel estimates are not applicable. The fact that the measure $\mu$ is not absolutely continuous with respect to the Haar measure prevents, {\it a priori}, any estimate of the form (\ref{eq:upperGaussianOffdiag}). Actually, the group generated by the support of the measure may very well be a dense finitely generated subgroup of exponential growth (these can already be found in $\R^2\rtimes O(2)$).  Therefore, even obtaining inequalities such as (\ref{eq:concentration}) for such measures requires a completely different approach. Indeed, we develop new direct arguments to bound tails of exit times, using the filtrations on nilpotent Lie groups.

\subsection{Concentration inequalities for random walks on groups of polynomial growth}

We consider random walks on groups of polynomial growth, endowed with word norms. We establish two different concentration inequalities, depending whether the step distribution $\mu$ is centred or not. Let $\mu$ be a Borel probability measure on a locally compact group $G$. Denote $H$ the subgroup generated by the support. Recall that the measure $\mu$ is centred if the mean value of its push-forward onto the abelianisation of $H$ vanishes, i.e. $\E(\pi_\ast\mu)=0$ with abelianisation map $\pi:H \to H/[H,H]\otimes \mathbb{R}$. Symmetric measures are centred.

The main results of this paper are the following maximal concentration
inequalities.

\begin{theorem}\label{thmIntro:NbyQ} 
Let $G$ be a locally compact, compactly generated group of polynomial
growth. Let $S$ be a compact symmetric generating set, and $\mu$ be a compactly supported measure on~$G$. Denote $(w_n)_{n\ge 0}$ the associated random walk starting at the neutral element.
\begin{itemize}
\item If $\mu$ is centred, there exists  constants $c_1,c_2>0$ such for all $t>0$
that 
\[
\sup_{n\in \mathbb{N}}\mathbb{P}\left(\max_{k\leq n}|w_{k}|_{S}\geq tn^{\frac{1}{2}}\right)\leq c_2\exp(-c_1t^{2}).
\]
\item In general, there exists an integer $s\ge 1$, constants $c_3,c_4>0$, and a trajectory $(g_n)_{n\in G}$ in $G$ depending only on $\mu$ such that for all $t>0$
\[
\sup_{n\in \mathbb{N}}\mathbb{P}\left(\max_{k\leq n}|w_{k}g_k^{-1}|_{S}\geq tn^{\frac{2s-1}{2s}}\right)\leq c_4\exp(-c_3t^{\frac{2s}{2s-1}}).
\]
\end{itemize}
\end{theorem} 

The case where
$G$ is abelian is well-known by martingale concentration, but to
the best of our knowledge the first part of above theorem is new even for symmetric
random walks on the continuous Heisenberg group, when the random walk step distribution is {\it not assumed} to be supported on a discrete subgroup or have a density with respect to the Haar measure. For centred random walks absolutely continuous with respect to the Haar measure on groups with polynomial growth, we have seen that these follow from the work of Alexopoulos (see \cite[(9.1)]{AlexopoulosLC} for the explicit statement).

In the non-centred case, we do not know of any known inequalities of this kind, even for a measure supported on the discrete Heisenberg group.

The integer $s$ and the trajectory $(g_n)$ are described in terms of the structure of the group $G$ and . By Losert~\cite{Losert}, every compactly generated locally compact group of polynomial growth admits a proper morphism with cocompact image into a semi-direct product of the form $N\rtimes Q$, where $N$ is a simply connected nilpotent group and $Q$ is a compact group. The integer $s$ is simply the degree of nilpotency of $N$. The trajectory $(g_n)$ is a lift to $G$ of a sequence $\exp(nv_\mu)$ on $N$, for some vector $v_\mu$ in the Lie algebra depending only on $\mu$. We call this vector the essential average of $\mu$. When $G=N$, this essential average is a lift of the mean value of the projection of $\mu$ to the abelianisation of the Lie algebra.
One can interpret $g_n$ as a coarse average of the distribution $\mu^{\ast n}$.

Provided that the measure $\mu$ is supported on the generating set $S$, the constants $c_1,c_2,c_3,c_4$ in Theorem~\ref{thmIntro:NbyQ} depend only on $(G,S)$ and a spectral constant $\kappa_\mu$ related to the operator $\Ad(\pi_{Q\ast} \mu)$ on the Lie algebra (see Lemma~\ref{lem:spectralgap} for the precise definition). This dependence in the spectral constant is necessary. For instance, if a drifted random walk on the euclidean plane is rotated with a very small probability, then its speed appears as linear for some interval of time, see Remark~\ref{rk:balistic}.

 The heart of the proof of Theorem~\ref{thmIntro:NbyQ} is the case of semi-direct products $N\rtimes Q$. The essential average $v_\mu \in \mathfrak{n}$ plays an important role. We associate to it the weighted filtration $\mathfrak{F}_{v_\mu}$, which was introduced by Bénard and Breuillard in order to obtain central and local limit theorems for non-centred random walks on simply connected nilpotent Lie groups~\cite{BeBr23,BeBr25}. For a homogeneous norm $|\cdot|_{\mathfrak{F}_{v_\mu}}$ adapted  to this filtration, we obtain subgaussian concentration away from the essential average.
 
 \begin{theorem}\label{thm:normIntro}
 Let $\mu$ be a compactly supported measure on $G=N \rtimes Q$. Denote $w_n$ the random walk at time $n$. Then there exists $c_1,c_2>0$ such that
\[
\sup_{n\in \mathbb{N}} \mathbb{P} \left( \frac{1}{\sqrt{n}}\max_{k\le n}\left(\left|w_k\exp\left(-kv_{\mu}\right)\right|_{\mathfrak{F}_\mu}\right) \ge t\right) \le c_2e^{-c_1t^2}.
\]
 \end{theorem}

With a slight abuse of notation, we write $|g|_{\mathfrak{F}_{v_\mu}}$ for $|\pi_{\mathfrak{n}}(g)|_{\mathfrak{F}_{v_\mu}}$. The constants $c_1,c_2$ depend only on the dimension and degree of nilpotency of the Lie group $N$, and on the measure $\mu$ via  the radius of the support, and the spectral constant $\kappa_\mu$ .

The proof of Theorem~\ref{thm:normIntro} is by double induction. First on the degree of nilpotency. When $N$ is abelian, the statement is nothing but the Azuma-Hoeffding inequality as the induced process on $N$ is a martingale with bounded increments. The case of two-step nilpotent groups is obtained using a Laplace transform characterisation of concentration inequalities. By homogeneity of the norm, one has to check that square roots of coordinates in the center are subgaussian, or equivalently that coordinates are subexponential. It is sufficient to do this for times powers of $2$. The process at time $2n$ is a sum of two processes at time $n$, plus a commutator term from the coordinates of the abelianisation, which is controlled by induction. The key point is that $1/\sqrt{2}$ times a sum of $c$-subexponential variables is still $c$-subexponential. As the coefficient $1/\sqrt{2}$ is strictly above $1/2$, we can write the process as an average of $c'$-subexponential variables for large enough $c'$. The case of degree of nilpotency from $3$ and above is treated in a similar though somewhat simpler manner, using the characterisation of concentration in terms of $p$-moments for all $p$ and the Baker-Campbell-Hausdorff formula.

\subsection{Uniform constants in finite-by-nilpotent groups}

The concentration constants in Theorem~\ref{thm:normIntro} depend only on the dimension and degree of nilpotency of $N$ and the spectral constant $\kappa_\mu$ of $\pi_{Q\ast}\mu$. When we induce these inequalities to a group $G$ of polynomial growth, new dependence appears in two ways from the relationship between $G$ and the related semi-direct product. Firstly one has to mod out by the maximal compact subgroup $K$ of $G$, and secondly one needs to embed $G/K$ cocompactly into a group of the form $N\rtimes Q$. The first dependence can by no means be made uniform since the compact group can be arbitrary. Even the second dependence cannot be removed : among co-compact subgroups of $N\rtimes Q$ the concentration constants cannot be uniform, if one uses word distances of the subgroups. For instance consider the symmetric random walk on the real line driven by the uniform measure on $2^{-j}$, $j\in\{0,\cdots,\ell\}$ and there opposites, $\mu_\ell=\sum_{i=0}^\ell \pm \delta_{2^{-i}}/(2\ell+2)$. With respect to the generating set $\{\pm2^{-j}, j=0,\cdots,\ell\}$, the $\mu_{\ell}$-random walk admits a linear drift at times below $\ell$ by comparison with a hypercube. 

However we show that for finite-by-nilpotent groups this dependence can be removed if we consider concentration with respect to an appropriate word norm on $N\rtimes Q$. Regarding the previous example, it is clear that by the Bernstein inequality, random walks driven by $\mu_\ell$ above satisfy Gaussian concentration with uniform constants if measured with the absolute value on $\mathbb{R}$.

\begin{theorem}
Let $1 \to N \to G \to F \to 1$ be an exact sequence with $N$ a simply connected nilpotent Lie group and $F$ a finite group. There exists an integer $k$ depending only on the dimension and degree of nilpotency of $N$ and on the cardinality of $F$ such that the following holds. For any compactly supported measure $\mu$ on $G$. There exists a compact generating set $T$ of $G$, depending only on $\mathrm{supp}(\mu)^k$, such that the concentration results of Theorem~\ref{thmIntro:NbyQ} hold for the $T$-word norm with constants $c_1,c_2,c_3,c_4$ depending only on the dimension and degree of nilpotency of $N$ and the spectral constant $\kappa_\mu$.
\end{theorem}

The key to this uniformity is a result of independent interest asserting that the exact sequence splits in a quantitative way (see Theorem \ref{thm:QuantSplit}). 

\subsection{Application to extensions by approximately ultrametric subgroups}

\begin{definition}[Maximally diffusive random process] 
For a random process $(w_{n})_{n\geq0}$ with values in a metric space
$(X,d)$, we say that its maximal displacements satisfy subgaussian concentration
with diffusive scaling if the family of random
variables $\left(\frac{1}{\sqrt{n}}\max_{i\leq n}d(w_{0},w_{i}\right)_{n\in \mathbb{N}})$ is subgaussian. For brevity, we will refer to this
property as \emph{maximally diffusive} in what follows. 
\end{definition}

It is well-known that the Gaussian
concentration above is equivalent to 
\[
\E_{\mu}\left(\max_{i\leq n}d(w_{0},w_{i})^{p}\right)^{1/p}\leq C(pn)^{1/2},
\]
for some constant $C>0$ and all $p\ge 2$ (see \S \ref{sec:prelimConcentration}).
Note that the terminology of \emph{diffusive} random walk displacement often
means the first moments satisfies $\E_{\mu}d(w_{0},w_{n})\leq Cn^{1/2}$,
we emphasize that in this article, maximally diffusive means subgaussian
concentration with diffusive scaling for maximal displacements, which
controls all $p$-moments.

We can now rephrase the first part of Theorem~\ref{thmIntro:NbyQ} as follows: any compactly
supported centred random walk on a locally compact group $G$ of
polynomial growth is maximally diffusive, where $G$ is equipped with
a word distance. An important feature of Theorem~\ref{thmIntro:NbyQ}
is that it will (almost automatically) upgrade to centred random walks on
a much larger class of groups. 




A key observation is that being maximally diffusive is stable
by extension by a normal subgroup which is sufficiently distorted.
This observation is at the heart of a result of Russ Thompson~\cite{Thompson}, who proved that simple random walks on torsion free polycyclic groups are diffusive (first moment estimate) provided that they are extensions of torsion free abelian groups by finitely generated nilpotent exponentially distorted subgroups. This raised the question whether all polycyclic groups have this property, which we answer positively.
This motivates the following definition.

\begin{definition}[AU subgroups] Let $E$ be a closed subgroup of a locally
compact, compactly generated group $G$. Let $S$ be a symmetric compact generating set 
of $G$ and $|\cdot|_S$ the associated word distance. We say that $E$ equipped with the induced metric 
$d_S$ is approximately ultrametric with constant multiplicative distortion and $O(\log n)$ additive distortion
if  there exists $C\geq1$ such that 
\[
|h_{1}h_{2}\ldots h_{n}|_{S}\leq C\log n+C\max_{i}|h_{i}|_{S},
\]
for all $h_{1},\ldots h_{n}\in E$. 
We say $E$ is an AU-subgroup of $G$ is 
this property holds for some (equivalently any) symmetric compact generating set of $G$
\end{definition} 

Note that when
$E$ is compactly generated, this definition recovers the usual notion
of exponential distortion. However in our applications, $E$ can be taken to
be a unipotent group over a finite product of local fields, and as
such, will be not be compactly generated as soon as at least one of
the fields is non-archimedean (for instance $\Q_{p}$).

We also note that for applications to displacement estimates, the requirement $O(\log n)$ on the additive error term can be relaxed to $o(n^{1/2})$, see Proposition~\ref{prop:EDconc}.
However we are not aware of situations where the ultrametric approximation has additive error between $\log n$ and $n^{1/2}$.

\begin{proposition}\label{prop:Edoesnotcount} Let $G$ be a locally
compact, compactly generated group. Assume that there exists a closed
normal subgroup $E$ that is an AU-subgroup in $G$. Assume
that any centred compactly generated random walk on $G/E$
is maximally diffusive, then the same holds for $G$. \end{proposition}

This key proposition is essentially contained in \cite[Theorem 4]{Thompson} where it is stated for $G/E$ free abelian. It motivates the introduction of the following class of
groups.

\begin{definition}[Polynomial-by-AU] A locally compact, compactly
generated group $G$ is polynomial-by-AU if there exists a short exact
sequence of locally compact groups $1\to E\to G\to P\to1$ such that
$E$ is AU, and $P$ has polynomial growth. \end{definition}

We can now state our main application. 

\begin{theorem}\label{thmIntro:EbyP} 

Let $G$ be a polynomial-by-AU group. Then any centred, compactly
supported random walk on $G$ is maximally diffusive.

\end{theorem}

We also refer to Proposition~\ref{prop:EDconc} for a slightly more general statement.
An important class of polynomial-by-AU groups is that of amenable
almost connected Lie groups. This was first proved by Guivarc'h \cite[Proposition 5]{Guivarch80} and later rediscovered by Osin \cite{Osin}. Hence we deduce

\begin{corollary}\label{cor:amenLie} Let $G$ be an almost connected amenable Lie group.
Then a centred, compactly supported random walk on $G$ is maximally
diffusive. \end{corollary} 

Since polycyclic groups are virtually uniform lattices in connected
solvable Lie groups, we immediately deduce that finitely supported
centred random walks on polycyclic groups are maximally diffusive. Symmetric random walks with finite second moment on polycyclic groups (even on groups quasi-isometric to polycyclic) were known to be diffusive by Peres-Zheng~\cite{PeresZheng}.

More generally the conclusion holds for solvable groups with finite Pr\"ufer ranks.
\begin{corollary}\label{cor:finireRank} Let $\Gamma$ be a finitely generated group which
is virtually a solvable group of finite Prüfer rank. Then a centred
random walk of finite support on $\Gamma$ is maximally diffusive.
\end{corollary}
The proof of this corollary first relies on a deep result of Kropholler and Lorensen \cite{KroLor}, which reduces the problem to virtually torsion-free such groups. The second step of the proof exploits a result of \cite{CT} which allows to realize our group as a uniform lattice in a locally compact group which resembles a Lie group defined over a finite product of local fields of zero characteristic (see \S \ref{sec:FiniteRank} for a precise definition of the class $\mathfrak{C}''$). 

\subsection{Law of iterated logarithm}

Let us finally mention that concentration inequalities yield laws of iterated logarithms for random walk trajectories. This is well known, established by Hebisch and Saloff-Coste for symmetric bounded densities random walks on groups of polynomial growth~\cite{HebishSaloff}, by Revelle on some lamplighter groups, solvable Baumslag-Solitar groups and Sol~\cite{Revelle}, by Thompson for some torsion-free polycyclic groups~\cite{Thompson}. Various laws of iterated logarithms have also been obtained by Amir and Blachar on diagonal products of lamplighter groups~\cite{AmirBlachar}.

\begin{proposition}
Assume a concentration inequality along a deterministic trajectory $(g_n)$ of the form
\[
\sup_{n\in \mathbb{N}}\mathbb{P}\left(\max_{k\leq n}|w_{k}g_k^{-1}|_{S}\geq tn^{\alpha}\right)\leq c_2\exp(-c_1t^{\beta}).
\]
Then the random walk satisfies an $\alpha$-law of iterated logarithm, i.e. there almost surely exists a constant $C$ such that
\[
\forall n \in \mathbb{N}, \quad |w_ng_n^{-1}|_S \le C\left(n \log \log n\right)^\alpha.
\]
\end{proposition}

\begin{proof}
Taking $n=2^j$ and $t=c_3(\log j)^\alpha$, we have for all $j$
\[
\mathbb{P}\left(\max_{k\leq 2^j}|w_{k}g_k^{-1}|_{S}\geq c_3(2^j\log j)^\alpha\right)\leq c_2j^{-c_1c_3^\beta}.
\]
This is summable if $c_3$ is large enough. By the Borel-Cantelli lemma, we have almost surely $\max_{2^{j-1} \le n \le 2^j} |w_ng_n^{-1}|_S\le c_3(2^j\log j)^\alpha$ for $j$ large. The conclusion follows.
\end{proof}

The proposition together with Theorem~\ref{thmIntro:NbyQ}, Corollaries~\ref{cor:amenLie} and~\ref{cor:finireRank} imply the following.

\begin{corollary}
Compactly supported centred random walks on groups of polynomial growth, on almost connected amenable Lie groups, on virtually solvable groups of finite Prüfer rank satisfy a $1/2$-law of iterated logarithm. Compactly supported random walks on groups of polynomial growth satisfy a $\frac{2s-1}{2s}$-law of iterated logarithm where $s$ is the degree of nilpotency of the associated nilpotent Lie group.
\end{corollary}

\begin{ack} We thank Emmanuel Breuillard for pointing to us the reference~\cite{Losert}.
	\end{ack}

\section{Review of concentration inequalities}
\label{sec:prelimConcentration} 

Concentration inequalities are bounds on the decay at infinity of random variables. For instance, martingales with bounded increments have a subgaussian tail decay once normalised by $\sqrt{n}$ according to the following classical result, see e.g.~\cite[Chapter 3]{BDR}.

\begin{proposition}[Azuma-Hoeffding inequality]\label{prop:AH}
Let $(X_n)_{n \in \mathbb{N}}$ be a martingale with increments in a bounded interval $[-M,M]$, then for all $n \in \mathbb{N}$ and all $t>0$ we have
\[
\mathbb{P}\left(\frac{|X_n|}{\sqrt{n}} \ge t\right) \le \exp\left(-\frac{t^2}{2M^2}\right).
\]
\end{proposition}


We will be interested in concentration inequalities with different exponents on $t$.

\begin{definition}\label{def-concentration}
Let
$0< \alpha<\infty$. A family $f=(f_{i})_{i\in I}$
of real random variables is \emph{$\alpha$-concentrated} if there exist constants $c_1,c_2>0$ such that for all $t>0$, 
\begin{align}\label{def:concentration}
\sup_{i\in I}\mathbb{P}(|f_{i}|\geq t)\leq c_2\exp(-c_1t^{\alpha}).
\end{align}
We also say \emph{subgaussian} for $2$-concentrated and \emph{subexponential} for $1$-concentrated.
\end{definition}

We record two elementary facts.
\begin{fact}\label{fact:concentration-moments}
Let $f=(f_{i})_{i\in I}$ and $a>0$. Denote
by $|f|^{a}:=(|f_{i}|^{a})_{i\in I}$. Then $f$ is $\alpha$-concentrated if and only if
$|f|^{a}$ is $\alpha/a$-concentrated.

In
particular, $f$ is $\alpha$-concentrated if and only if $|f|^{\alpha/2}$
is subgaussian.
\end{fact}

\begin{fact}\label{def:vectorvalued}
Let $V$ be a Euclidean vector space of finite dimension $d$. Let $f=(f_{i})_{i\in I}$
be a family of $V$-valued random variables. The following are equivalent:
\begin{enumerate}
\item the family $(\|f_{i}\|)_{i}$ is $\alpha$-concentrated; 
\item the family $(\langle f_{i},u\rangle)_{i,u}$ is $\alpha$-concentrated,
where $u$ runs through all unit vectors of~$V$. 
\end{enumerate}
\end{fact}

If one of these conditions is satisfied, then $f$ is called $\alpha$-concentrated. 

\begin{proof}[Proof of Fact~\ref{def:vectorvalued}]
Assume the second point and denote $(u_j)_{j=1}^d$ an orthonormal basis of $V$. As $\|f_i\| \le \sqrt{d}\sup_j |\langle f_i,u_j\rangle|$, we have
\begin{align*}
\mathbb{P}\left(\|f_i\| \ge t\right) & \le \mathbb{P}\left(\sup_j  |\langle f_i,u_j\rangle| \ge \frac{t}{\sqrt{d}} \right) \le d \sup_{i,j} \mathbb{P}\left(  |\langle f_i,u_j\rangle| \ge \frac{t}{\sqrt{d}} \right) \le dc_2 \exp\left(-\frac{c_1}{d^{\frac{\alpha}{2}}}t^\alpha \right).
\end{align*}
\end{proof}

Note that in the reverse implication, the constants involved depend on the dimension~$d$.

\subsection{Characterisations of concentration inequalities}

We recall the classical characterisation in terms of moments -- see for instance~\cite[Lemma~3.1]{Pisier}. 

\begin{proposition}\label{prop:p-ConcentrationInequalities} Let
$0\leq\alpha<\infty$. A family $f=(f_{i})_{i\in I}$ of real random variables is $\alpha$-concentrated if and only if 
 there exists a constant $C$ such that for all $p\geq2$, 
\[
\sup_{i\in I}\|f_{i}\|_{p}\leq Cp^{1/\alpha}.
\]
\end{proposition}

We shall also use the following well-known characterisations of
subexponential and subgaussian (families of) random variables in terms
of Laplace transform, see for instance~\cite[Propositions 2.5.2 and 2.7.1]{Vershynin}.

\begin{proposition}\label{prop:SubexpSubGauss} A family $f=(f_{i})_{i\in I}$
of random variables is 
\begin{itemize}
\item subgaussian if and only if $\sup_{i\in I}|\E(f_{i})|<\infty$ and there exists a constant $K$ such that for all $t \in \mathbb{R}$,
\begin{align}\label{eq:concentLaplace}
\sup_{i\in I}\E\left(\exp(t(f-\E(f))\right)\leq\exp(Kt^{2}).
\end{align}
\item subexponential  if and
only $\sup_{i\in I}|\E(f_{i})|<\infty$ and there exist constants
$(K,b)$  such that (\ref{eq:concentLaplace}) holds  for all $|t|<b$.
\end{itemize}
\end{proposition} 

In the sequel we will say that $f=(f_{i})_{i\in I}$ is $(E,K,b)$-subexponential
if $\sup_{i\in I}\E(f_{i})\leq E$, and the statement of the proposition
is satisfied. If the variables are centred, then we shall call it
$(K,b)$-subexponential. We record the following. 

\begin{remark}\label{rk:lambda}
If a family $f$ of random variables is $(E,K,b)$-subexponential, then for any real parameter $\lambda$, the familly $\lambda f$ is $(|\lambda| E, \lambda^2 K, \frac{b}{|\lambda|})$-subexponential.
\end{remark}

\begin{remark}
Although we do not state it explicitly in the characterisations of concentration above, it is true that the constants $c_1,c_2$ of Definition~\ref{def-concentration}, the constant $C$ of Proposition~\ref{prop:p-ConcentrationInequalities} and the constants $K,b$ of Proposition~\ref{prop:SubexpSubGauss} depend explicitly on each others, see~\cite{Pisier,Vershynin}.
\end{remark}

\subsection{Stability under operations}

It follows from Proposition~\ref{prop:p-ConcentrationInequalities} that a convex combination of $\alpha$-concentrated families is still $\alpha$-concentrated.

Using Jensen inequality in the Laplace characterisation of Proposition~\ref{prop:SubexpSubGauss} gives immediately the quantitative statement.

\begin{corollary}\label{cor:convexSubExp} 
Let $f=(f_{i})_{i\in I}$ and $h=(h_{i})_{i\in I}$ be two families
of random variable and let $\eta\in[0,1]$. If $f$ and $h$ are $(E,K,b)$-subexponential,
then so is $\eta f+(1-\eta)h$.
\end{corollary}

 We now turn to products of random variables. 
For $f=(f_{i})_{i\in I}$ and $h=(h_{i})_{i\in I}$,  denote
 $fh=(f_{i}h_{i})_{i\in I}$ the product family.

\begin{proposition}\label{prop:products} Let $f=(f_{i})_{i\in I}$
and $h=(h_{i})_{i\in I}$ be two families of real random variables.
If $f$ is $\alpha$-concentrated and $g$ is $\beta$-concentrated
for some $\alpha,\beta>0$, then $fh$ is $\gamma$-concentrated for
$\gamma=\frac{\alpha\beta}{\alpha+\beta}$. 

In particular, if both
$f$ and $h$ are subgaussian, then their product is subexponential.
\end{proposition} 
\begin{proof}
Note that we can assume that $f$ and $h$ are positive. By the AM-GM
inequality, we have 
\[
(fh)^{\gamma/2}=(f^{\alpha/2})^{\frac{\beta}{\alpha+\beta}}(h^{\beta/2})^{\frac{\alpha}{\alpha+\beta}}\leq\frac{\beta}{\alpha+\beta}f^{\alpha/2}+\frac{\alpha}{\alpha+\beta}h^{\beta/2}.
\]
As observed above any convex combination of two subgaussian
families of variables is subgaussian, hence $(fh)^{\gamma/2}$ is
subgaussian. By Fact~\ref{fact:concentration-moments}, the proposition is proved. 
\end{proof}

An immediate computation yields the following.
\begin{corollary}\label{cor:2/c}
Let $f^{(1)},\dots,f^{(r)}$ be families of real random variables such that $f^{(j)}$ is $\frac{2}{c_j}$-concentrated for all $1 \le j \le r$. Then the product $f^{(1)}\dots f^{(r)}$  is $\frac{2}{\sum_{j=1}^r c_j}$-concentrated.
\end{corollary}

Finally Doob's maximal inequalities permit to upgrade concentration of a martingale to its maximum.

\begin{proposition}\label{prop:MaxConcentrated} Let $0<\alpha<\infty$,
and let $(X_{n})$ be a martingale. Denote $X_{n}^{*}=\max_{1\leq k\leq n}X_{n}$.
Then if $(X_{n})$ is $\alpha$-concentrated, then so is $(X_{n}^{*})$
\end{proposition} 
\begin{proof}
By Doob's maximal $L^{p}$-inequality (see for instance \cite[Corollary II-1-6]{Revuz-Yor}),
we have for all $p>1$, 
\[
\|X_{n}^{*}\|_{p}\leq\frac{p}{p-1}\|X_{n}\|_{p}.
\]
In particular, $\|X_{n}^{*}\|_{p}\leq2\left\Vert X_{n}\right\Vert _{p}$
for all $p\geq2$, and we conclude thanks to Proposition~\ref{prop:p-ConcentrationInequalities}. 
\end{proof}

\section{Filtrations and norms on nilpotent Lie groups}\label{sec:filtration}

The \emph{lower central series} of a group $N$ is the sequence of
characteristic subgroups inductively defined by $\gamma_{1}(N)=N$
and $\gamma_{i+1}(N)=[N,\gamma_{i}(N)]$ for $i\geq1$. The group
$N$ is called \emph{$s$-nilpotent} if $\gamma_{s}(N)\neq\left\{ 1\right\} $
and $\gamma_{s+1}(N)=\left\{ 1\right\} $.

Similarly, a Lie algebra $\mathfrak{n}$ is $s$-nilpotent if its
lower central series defined by $\gamma_{1}(\mathfrak{n})=\mathfrak{n}$
and $\gamma_{i+1}(\mathfrak{n})=[\mathfrak{n},\gamma_{i}(\mathfrak{n})]$
for $i\geq1$ satisfies $\gamma_{s}(\mathfrak{n})\neq\left\{ 0\right\} $
and $\gamma_{s+1}(\mathfrak{n})=\left\{ 0\right\} $.


We let $N$ be a simply connected nilpotent Lie group, $\mathfrak{n}$
be its Lie algebra. We shall assume that $N$ or equivalently $\mathfrak{n}$
is $s$-step nilpotent.  We identify the Lie group $N$ with its Lie algebra
$\mathfrak{n}$ via the exponential map $\exp:\mathfrak{n}\to N$. 

We shall also assume that $N$ is acted upon continuously by a compact Lie group $Q$.
Our main object of consideration will be the semi-direct product $G=N\rtimes Q$.

\subsection{Filtrations on nilpotent Lie groups}

We consider on $\mathfrak{n}$ a specific kind of filtrations, introduced by Bénard and Breuillard to obtain central limit theorems on nilpotent Lie groups~\cite{BeBr23}.

\begin{definition}\label{not:Fv}
Given a vector $v \in \mathfrak{n}$, we let $\mathfrak{F}_v$ be the filtration  given by $\mathfrak{n}^{(0)}=\mathfrak{n}^{(1)}=\mathfrak{n}$ and for $i \ge 1$
\[
\mathfrak{n}^{(i+1)}= [\mathfrak{n},\mathfrak{n}^{(i)}]+[v,\mathfrak{n}^{(i-1)}].
\]
When $v=0$, this is simply the lower central series. 
\end{definition}

Observe that the filtration $\mathfrak{F}_v$ is preserved by the $Q$ action as soon as $v$ is fixed by $Q$. We record elementary facts about such filtration.

\begin{proposition}\label{prop:nested} The nested
sequence of ideals 
\[
\mathfrak{n}=\mathfrak{n}^{(1)}\supseteq\mathfrak{n}^{(2)}\supseteq\ldots\supseteq\mathfrak{n}^{(c)}\supseteq\mathfrak{n}^{(c+1)}=\{0\},
\]
has the following properties.
\begin{itemize}
\item $\mathfrak{n}^{(i+j)}\supseteq\left[\mathfrak{n}^{(i)},\mathfrak{n}^{(j)}\right]$ for all $i,j$.
\item  $\mathfrak{n}^{(i)} \supseteq \gamma_i(\mathfrak{n})$ for all $i$.
\item $\gamma_{i+1}(\mathfrak{n})  \supseteq \mathfrak{n}^{(2i)}$ for all $i$. In particular, the sequence terminates for some $s\le c \le 2s$ where $s$ is the degree of nilpotency of $\mathfrak{n}$.
\item The filtration depends only on the class of $v$ in $\mathfrak{n}/[\mathfrak{n},\mathfrak{n}]$.
\end{itemize}
\end{proposition}

\begin{proof}
For the first point, observe that it is true for $i=1$ or $j=1$. We proceed by induction on $k=i+j$. It holds for $k=2$. We further induce on $\ell=\min(i,j)$ and assume $j \le i$. We have by definition and Jacobi identity
\begin{align*}
[\mathfrak{n}^{(i)},\mathfrak{n}^{(j)}] & = [\mathfrak{n}^{(i)},[\mathfrak{n},\mathfrak{n}^{(j-1)}]] + [\mathfrak{n}^{(i)},[v,\mathfrak{n}^{(j-2)}]] \\
& \subset [\mathfrak{n},[\mathfrak{n}^{(j-1)},\mathfrak{n}^{(i)}]]+[\mathfrak{n}^{(j-1)},[\mathfrak{n}^{(i)},\mathfrak{n}]]+[v,[\mathfrak{n}^{(j-2)},\mathfrak{n}^{(i)}]]+[\mathfrak{n}^{(j-2)},[\mathfrak{n}^{(i)},v]].
\end{align*}
By induction on $k$, we have $[\mathfrak{n}^{(j-1)},\mathfrak{n}^{(i)}]\subset  \mathfrak{n}^{(i+j-1)}$ which treats the first bracket. As $[\mathfrak{n}^{(i)},\mathfrak{n}]\subset \mathfrak{n}^{(i+1)}$, the second bracket is treated by induction on $\ell$. By induction on $k$, we have $[\mathfrak{n}^{(j-2)},\mathfrak{n}^{(i)}] \subset \mathfrak{n}^{(i+j-2)}$, treating the third bracket. As $[\mathfrak{n}^{(i)},v]\subset \mathfrak{n}^{(i+2)}$, the fourth bracket is treated by induction on $\ell$.

The second and third points are proved by induction on $i$. We have $\gamma_{i+1}(\mathfrak{n})=[\mathfrak{n},\gamma_i(\mathfrak{n})] \subset [\mathfrak{n},\mathfrak{n}^{(i)}] \subset \mathfrak{n}^{(i+1)}$. Also there is an obvious inclusion
$\mathfrak{n}^{(2i+2)}=[\mathfrak{n},\mathfrak{n}^{(2i+1)}]+[v,\mathfrak{n}^{(2i)}]\subset [\mathfrak{n},\mathfrak{n}^{(2i)}]$ so by induction $\mathfrak{n}^{(2i+2)} \subset [\mathfrak{n},\gamma_{i+1}(\mathfrak{n})]=\gamma_{i+2}(\mathfrak{n})$. 

For the last point, let $\mathfrak{n}_1^{(i)}$ and $\mathfrak{n}_2^{(i)}$ denote filtrations associated to $v_1$ and $v_2$ respectively with $v_1 \in v_2+[\mathfrak{n},\mathfrak{n}]$. Assume by induction that $\mathfrak{n}_1^{(j)}=\mathfrak{n}_2^{(j)}$ for all $j\le i$. Then
\[
\mathfrak{n}_1^{(i+1)}=[\mathfrak{n}_1,\mathfrak{n}_1^{(i)}]+[v_1,\mathfrak{n}_1^{(i-1)}] \subset [\mathfrak{n}_2,\mathfrak{n}_2^{(i)}]+[v_2,\mathfrak{n}_2^{(i-1)}]+[[\mathfrak{n},\mathfrak{n}],\mathfrak{n}_2^{(i-1)}]\subset \mathfrak{n}_2^{(i+1)}.
\]
Equality holds by symmetry.
\end{proof}

The behaviour of brackets with respect to the filtration is given by the next lemma.

\begin{lemma}\label{lem:commute}
Given a filtration $\mathfrak{F}_v$ as in Definition~\ref{not:Fv}, consider $u_1,\dots, u_m \in \mathfrak{n}$. Assume that there are $k$ indices such that $u_j \in \mathbb{R}v$ and $\ell_i$ remaining indices such that $u_j \in \mathfrak{n}^{(i)}$. Then 
\[
[u_1,\dots,[u_{m-1},u_m]] \in \mathfrak{n}^{(p)} \quad \textrm{for }p=2k+\sum_{i\ge 1}i\ell_i,
\] 
unless $m=1$ and $u_1\in \mathbb{R}v$.
\end{lemma}

\begin{proof}
Proceed by induction on $m$. This is obvious for $m=1$. Assume $w:=[u_2,\dots,[u_{m-1},u_m]] \in \mathfrak{n}^{(q)}$. If $u_1 \in \mathfrak{n}^{(i)}$, then $[u_1,w] \in \mathfrak{n}^{(q+i)}$ by the previous proposition. If $u_1 \in \mathbb{R}v$, then $[u_1,w] \in \mathfrak{n}^{(q+2)}$ by definition of $\mathfrak{F}_v$. Of course $[v,v]=0 \in \mathfrak{n}^{(4)}$, so the exception for $m=1$ no longer holds for $m\ge 2$.
\end{proof}

\subsection{Norms adapted to the filtration}

Given a $Q$-invariant vector $v\in \mathfrak{n}$, the Lie algebra $\mathfrak{n}$ is equipped with two $Q$-invariant filtrations $\mathfrak{F}=\mathfrak{F}_{v}=(\mathfrak{n}_i)_{i=1}^c$ and $\mathfrak{L}=(\gamma_i(\mathfrak{n}))_{i=1}^s$ defined in the previous section. They coincide when $v=0$. 

Given a $Q$-invariant euclidean norm $\|\cdot\|_e$ on $\mathfrak{n}$, we denote $\mathfrak{m}_i$ the orthogonal complement of $\mathfrak{n}_{i+1}$ in $\mathfrak{n}_i$, and $\mathfrak{m}^i$ the orthogonal complement of $\gamma_{i+1}(\mathfrak{n})$ in $\gamma_i(\mathfrak{n})$. We get two decompositions in orthogonal direct sum:
\begin{align}\label{eq:decomp}
\mathfrak{n}=\mathfrak{m}_1 \oplus \dots \oplus \mathfrak{m}_c=\mathfrak{m}^1 \oplus \dots \oplus \mathfrak{n}^s.
\end{align}
We denote $u=u_1+\dots+u_c=u^1+\dots+u^s$ the associated decompositions of a vector $u \in \mathfrak{n}$.

Let $\varphi$ be a norm on the vector space $\mathfrak{n}$. We define
\begin{align}\label{def:norm}
|u|_{\mathfrak{F},\varphi}=\sup_{i\in\{1,\dots,c\}} \varphi(u_i)^{\frac{1}{i}} \quad \textrm{and} \quad |u|_{\mathfrak{L},\varphi}=\sup_{i\in\{1,\dots,s\}} \varphi(u^i)^{\frac{1}{i}}.
\end{align}
These norms are $Q$-invariant when $\varphi$ is $Q$-invariant.

The decompositions (\ref{eq:decomp}) induce two one-parameter families of adapted dilations $\left(D^{\mathfrak{F}}_{r}\right)_{r>0}$ and $\left(D^{\mathfrak{L}}_{r}\right)_{r>0}$, given by 
\[
D^{\mathfrak{F}}_{r}(x)=r^{i}x \textrm{ for } x\in\mathfrak{m}_{i} \quad \textrm{and} \quad D^{\mathfrak{L}}_{r}(x)=r^{i}x \textrm{ for } x\in\mathfrak{m}^{i}.
\]
Even though they need not be subadditive, we will call $|\cdot|_{\mathfrak{F},\varphi}$ and $|\cdot|_{\mathfrak{L},\varphi}$ homogeneous norms on the nilpotent Lie algebra $\mathfrak{n}$, because they satisfy:
\[
\left|D^{\mathfrak{F}}_{r}(x)\right|_{\mathfrak{F},\varphi}=r|x|_{\mathfrak{F},\varphi} \quad \textrm{and} \quad \left|D^{\mathfrak{L}}_{r}(x)\right|_{\mathfrak{L},\varphi}=r|x|_{\mathfrak{L},\varphi}.
\]
Such norms have played an important role in the study of nilpotent Lie groups. In Section~\ref{sec:NQ}, we will show subgaussian concentration for random walks with respect to such norms with some vector $v$ depending on the driving measure. Prior to that, we record the behaviour of norms under taking brackets, and the relationship between norms of the form $|\cdot|_{\mathfrak{L},\varphi}$ on $\mathfrak{n}$ and word norms on $N$.

The constant of bilinearity of a norm $\varphi$, defined as
\[
C_\varphi:=\sup_{u,v\in B_\varphi(1)} \varphi([u,v])
\]
plays an important role via the following fact.

\begin{fact}\label{lem:multicomm} For all $i_{1},i_{2},\ldots,i_{m}\geq1$, for every $u_{j}\in\mathfrak{m}_{i_{j}}$,
we have 
\[
\varphi\left([u_{1},[u_{2},\ldots,[u_{m-1},u_{m}]]]\right)\le C_\varphi^{m-1}\varphi(u_1)\ldots\varphi(u_m)=C_\varphi^{m-1}|u_{1}|_{\mathfrak{F},\varphi}^{i_{1}}\ldots|u_{m}|_{\mathfrak{F},\varphi}^{i_{m}}.
\]
\end{fact}
We stated the fact for the filtration $\mathfrak{F}_v$ but of course taking $v=0$, it also holds for the lower central series $\mathfrak{L}$.
\begin{proof}[Proof of Fact \ref{lem:multicomm}]
Bilinearity of the Lie bracket ensures that for all $u,v\in\mathfrak{n}$,
\[
\varphi([u,v])\le C_\varphi\varphi(u)\varphi(v).
\]
The fact is obtained by immediate induction, as $\varphi(u_j)=|u_j|_{\mathfrak{F},\varphi}^{i_j}$. 
\end{proof}

The following proposition is adapted from the work of Guivarch~\cite{Guivarch}, stated and proved in a form convenient to our purpose, keeping track of the constants involved.

We denote $B_\varphi(r)=\{u\in \mathfrak{n}: \varphi(u) <r\}$ the open ball of radius $r$ centred in $0$  for the norm $\varphi$, and the Lie product by
$u\ast v=\log(\exp(u)\exp(v))$ for all $u,v \in \mathfrak{n}$.

\begin{proposition}\label{prop:normsubadd}
Given a euclidean norm $\|\cdot\|_e$, there exists a norm $\varphi$ on $\mathfrak{n}$ with the following properties.
\begin{itemize}
\item There exists a constant $\kappa$, depending only on the bilinearity constant $C_{\|\cdot\|_e}$ of the euclidean norm, such that
\[
\forall v \in \mathfrak{n}, \quad \|v\|_e \le \kappa \varphi(v).
\]

\item $C_\varphi:=\sup\left\{ \varphi([u,v]) \le 1 : \varphi(u),\varphi(v)\le 1\right\}$.
\item The associated homogeneous norm $|\cdot|_{\mathfrak{L},\varphi}$ is subadditive
\[
\forall u,v \in \mathfrak{n}, \quad |u\ast v|_{\mathfrak{L},\varphi} \le |u|_{\mathfrak{L},\varphi}+|v|_{\mathfrak{L},\varphi}.
\]
\item For any compact generating set $R$ of $N$ containing a neigborhood of the identity, there exists constants $C_1,D_2\ge 1$ such that 
\[
B_{\|\cdot\|_e}\left(\frac{1}{D_2}\right) \cap \mathfrak{m}^1 \subset \log R \subset B_{|\cdot|_{\mathfrak{L},\varphi}}(C_1).
\]
Then we have
\[
\forall u \in \mathfrak{n}, \quad \frac{1}{C_1} |u|_{\mathfrak{L},\varphi}\le  |\exp(u)|_R \le C_2\left(|u|_{\mathfrak{L},\varphi} +1\right),
\]
where $C_2$ depends only on $C_1$, $D_2$ and the degree of nilpotency $s$.
\end{itemize}
Moreover $\varphi$ can be chosen such that its restriction to $\mathfrak{m}^1$ coincides with that of $\|\cdot\|_e$ and $|\cdot|_{\mathfrak{L},\varphi}$. The norm $\varphi$ can be chosen $Q$-invariant when $\|\cdot\|_e$ is $Q$-invariant.
\end{proposition}

The proposition is proved in Appendix~\ref{app:norms}. 
In fact, the homogeneous subadditive norm $|\cdot|_{\mathfrak{L},\varphi}$ is equivalent to a euclidean norm by~\cite{HS90}, but we do not pursue in this direction.

\subsection{The Baker-Campbell-Hausdorff formula}\label{sec:BCH}

We recall the Baker-Campbell-Hausdorff
formula in its explicit form \cite{Dynkin} : 
\[
\exp(x)\exp(y)=\exp\left(\sum_{n=1}^{\infty}\frac{(-1)^{n-1}}{n}\sum_{r_{1},s_{1},\dots r_{n},s_{n}}\frac{[x^{r_{1}}y^{s_{1}}\dots x^{r_{n}}y^{s_{n}}]}{\sum_{j=1}^{n}(r_{j}+s_{j})\prod_{i=1}^{n}r_{i}!s_{i}!}\right)
\]
where the integers in the sum satisfy $r_{i}+s_{i}>0$, and 
\[
[x^{r_{1}}y^{s_{1}}\dots x^{r_{n}}y^{s_{n}}]=\underbrace{[x,\dots,[x}_{r_{1}},\underbrace{[y,\dots,[y}_{s_{1}},[x,\dots]]]]].
\]

The important point is that the number of terms in the sum is finite, depending only on the degree of nilpotency $s$, and the coefficients are explicit. For instance the first terms are: 
\[
\exp(x)\exp(y)=\exp\left(x+y+\frac{1}{2}[x,y]+\frac{1}{12}[x,[x,y]]+\frac{1}{12}[y,[y,x]]+\dots\right).
\]

\section{Gaussian concentration for the adapted norm}\label{sec:NQ}

This section is devoted to the proof of gaussian concentration inequalities for random walks on semidirect products $N\rtimes Q$ of nilpotent Lie group and compact group with respect to a norm $|\cdot|_{\mathfrak{F}_v}$ associated to an adapted filtration $\mathfrak{F}_v$ as defined in Section~\ref{sec:filtration}.

\subsection{Random walks on semi-direct products}\label{subsec:notations}

We consider here a group $G=N \rtimes Q$, semi-direct product between $N$ simply connected nilpotent Lie group and $Q$ a compact Lie group. Any element decomposes as a product $g=x\kappa$ with $x \in N \times\{1\}$ and $\kappa \in \{1\}\times Q$. Letting $\xi=\log(x)\in \mathfrak{n}$, we will identify $G$ with $\mathfrak{n}\times Q$ and write $g=(\xi,\kappa)$. As $g_1g_2=\exp(\xi_1)\kappa_1 \exp(\xi_2) \kappa_1^{-1} \kappa_1 \kappa_2$, we have $g_1g_2=(\xi_1\ast\kappa_1\xi_2, \kappa_1\kappa_2)$, where $\kappa\xi:=\Ad(\kappa)\xi$ denotes the adjoint action and $\xi\ast\xi'=\log(\exp(\xi)\exp(\xi'))$ is the Lie product.

\begin{notation}
Given a $Q$-invariant euclidean norm $\|\cdot\|_e$ on $\mathfrak{n}$, we consider the norm $\varphi$ given by Proposition~\ref{prop:normsubadd}. We further equip the Lie algebra $\mathfrak{n}$ with the homogeneous norm $|\cdot|_{\mathfrak{F}_\mu,\varphi}$, or $|\cdot|_{\mathfrak{F}_\mu}$ in short, defined by (\ref{def:norm}) in the previous section. Note that this norm depends on  the weight filtration $\mathfrak{F}_\mu$, which is determined by the vector $v_\mu$ of Lemma~\ref{lem:conjugate}.

To ease notations, we denote by $\|\cdot\|$ the norm $\varphi$ in the remainder of this section. Note that by Proposition~\ref{prop:normsubadd}, its bilinearity constant satisfies $C_{\|\cdot\|}\le 1$ and its restriction to $\mathfrak{m}_1$ is euclidean.
\end{notation}

Let $\mu$ be a compactly supported probability measure on $G=N \rtimes Q$. We consider time $n$ random walk $w_n=g_1\dots g_n$ with $g_j=(\xi_j,\kappa_j)$ independent and $\mu$ distributed. 
Then for $n\geq 1$
\[
w_n=(\xi_1\ast\kappa_1\xi_2 \ast \dots \ast \kappa_1\dots\kappa_{n-1}\xi_n,\kappa_1\dots\kappa_n).
\]
Writing $w_n=(z_n,q_n)$ with $z_n \in \mathfrak{n}$ and $q_n \in Q$, we have
\[
w_{n}=(z_{n-1} \ast q_{n-1}\xi_n,q_{n-1}\kappa_{n}).
\]

Let $\mu_Q$ be the pushforward of $\mu$ by the projection $\pi_Q:N\rtimes Q \to Q$. Consider the decomposition of the abelianisation $\mathfrak{m}^1=V\oplus W$ where $V$ is the space of $\Ad(Q)$-invariant vectors and $W$ its orthogonal complement. As we may replace $Q$ with its smallest closed subgroup containing the semigroup generated by the support of $\mu_Q$, we may and do assume that the subspace $V$ is the set of vectors fixed by almost all elements in the support of $\mu_Q$.

\begin{lemma}\label{lem:spectralgap}
The restriction to $W$ of the operator $\mathrm{id}-\Ad(\mu_Q)$ is invertible. In particular, there exists $\kappa_\mu>0$ such that
\begin{equation}\label{eq:gap}
\forall w \in W, \quad \left\|\left(\mathrm{id}-\Ad(\mu_Q)\right)^{-1}w\right\| \le \frac{1}{\kappa_\mu}\|w\|.
\end{equation}
\end{lemma}

We call this number $\kappa_\mu$ the spectral constant of the operator $\Ad(\mu_Q)$ on $W$. Note that $\kappa_\mu$ is positive even when the operator does not admit a spectral gap. It could be for instance that $\mu_Q$ is a Dirac mass, in which case the spectrum of $\Ad(\mu_Q)$ is supported on the unit circle.

\begin{proof}
Let us denote $W_1$ the set of vectors of norm $1$ in $W$. Any vector in $W_1$ admits $\varepsilon>0$ and a neighbourhood all of whose elements $w$ satisfy $\mu_Q\{q \in Q: \|qw-w\|\ge \varepsilon\}\ge \varepsilon$. By compacity, this property holds with a uniform $\varepsilon>0$ for all $w$ in $W_1$. Let $p$ be the orthogonal projection onto the vector space spanned by $w$. For $q \in Q$, denote $p(qw)=\lambda(q)w$. Of course $\lambda(q)\le 1$. Moreover as the norm is euclidean, the inequality $\|qw-w\|\ge \varepsilon$ implies $\lambda(q) \le 1-\varepsilon^2/10$. Integrating with respect to $\mu_Q$, we get $\lambda(\Ad(\mu_Q))\le 1-\varepsilon^3/10$. This implies $\|w-\Ad(\mu_Q)w\|\ge \varepsilon^3/10$ for any $w \in W_1$, whence the lemma.
\end{proof}


We will also consider the minimal radius of a ball containing the $\mathfrak{n}$-projection of the support of $\mu$, which we denote by
\[
R_\mu:=\sup\left\{\|\pi_{\mathfrak{n}}(g)\|:g \in \mathrm{supp}(\mu)\right\},
\]
where $\pi_{\mathfrak{n}}:\mathfrak{n}\rtimes Q \to \mathfrak{n}$ is the natural projection.



\subsection{Centering the nilpotent part}

Consider the map $\pi_{\mathrm{ab}}^{G}: G \to G_{\mathrm{ab}}^\R$, where $G_{\mathrm{ab}}^\R=G/[G,G] \otimes \R$ is the ``torsion free abelianisation".. The drift of the random walk is the vector $\overline{v}_\mu:=\mathbb{E}(\pi_{\mathrm{ab}}^{G}\mu) \in G_{\mathrm{ab}}^\R$. The random walk is centred, or also $G$-centred, if $\overline{v}_\mu=0$.
Let $\pi_{\mathrm{ab}}^{\mathfrak{n}}: N \rtimes Q \to \mathfrak{m}_1 \subset \mathfrak{n}$ given by $\pi_{\mathrm{ab}}^{\mathfrak{n}}(\xi,\kappa)=\xi_1$, where $\xi=\xi_1+\dots+\xi_c$ is the decompostion (\ref{eq:decomp}). The random walk is $\mathfrak{n}$-centred if  $\mathbb{E}\pi_{\mathrm{ab}}^{\mathfrak{n}}\mu=0$. 
The $\mathfrak{n}$-centred random walks are $G$-centred, but the converse is true only up to conjugacy. For  $y \in \mathfrak{n}$, we denote by $c_y$ the conjugacy by $(y,1)$.

\begin{lemma}\label{lem:conjugate}
Let $\mu$ be a compactly supported probability measure on $G=N\rtimes Q$ as above. Then there exists $v_\mu \in \mathfrak{m}_1$ which is $\Ad(Q)$-invariant and there exists some $y \in \mathfrak{n}$ such that 
\[
\mathbb{E}\pi_{\mathrm{ab}}^{\mathfrak{n}}c_y\mu=v_\mu.
\]
The norms are bounded by $\|v_\mu\|\le R_\mu$ and  $\|y\|\le \frac{R_\mu}{\kappa_\mu}$.

When $\mu$ is $G$-centred, one has $v_\mu=0$.
\end{lemma}

Replacing $\mu$ by $c_y\mu$ only affects distances involving the random walk by an additive constant depending only on $R_\mu$ and $\kappa_\mu$. 

\begin{proof} Let $\pi_{\mathrm{ab}}^{\mathfrak{n}} : \mathfrak{n} \rtimes Q \to \mathfrak{m}^1= \mathfrak{n}/[\mathfrak{n},\mathfrak{n}]$ the projection to the abelianization, which we identify with the orthogonal complement $\mathfrak{m}_1$ of $[\mathfrak{n},\mathfrak{n}]$.
As 
\[
(-y,1)(x,q)(y,1)=(x-y+\Ad(q)y,q),
\]
we have $\mathbb{E}(\pi_{\mathrm{ab}}^{\mathfrak{n}}c_y\mu)=\mathbb{E}(\pi_{\mathrm{ab}}^{\mathfrak{n}}\mu)+\Ad(\mu_Q)y-y$.  Consider the orthogonal decomposition $\mathfrak{m}_1=V\oplus W$ with $V$ the $\Ad(Q)$ invariant vectors. Under this decomposition, we denote $\mathbb{E}(\pi_{\mathrm{ab}}^{\mathfrak{n}}\mu)= v_\mu+w_\mu$. Note that $\|v_\mu\| \le R_\mu$ and $\|w_\mu\|\le R_\mu$. Now take $y:=(\mathrm{id}-\Ad(\mu_Q))^{-1}(w_\mu)$. The bound on its norm follows from (\ref{eq:gap}). Then we have $\mathbb{E}(\pi_{\mathrm{ab}}^{\mathfrak{n}}c_y\mu)=v_\mu$. 

When moding out $\mathfrak{m}^1$ by $W$, we get a map $G\to V\times Q\to V$  sending $(\xi_1+\dots+\xi_c,q)$ to $\pi_V(\xi_1)$. As it factorises $G\to G_{\mathrm{ab}}^\R$, we get a linear map $\psi:G_{\mathrm{ab}}^{\mathbb{R}}\to V$. It follows that when $\mu$ is centred $v_\mu=\mathbb{E}(\pi_V(\mu))=\psi(\mathbb{E}\pi_{\mathrm{ab}}^{\mathbb{R}}\mu)=0$. 
\end{proof}

\subsection{Gaussian concentration for an adapted filtration norm}

To our probability measure $\mu$ on $G$, we associate the vector $v_\mu \in \mathfrak{n}$ given by Lemma~\ref{lem:conjugate}. In particular, $v_\mu=0$ when $\mu$ is centred.
Consider as in Definition~\ref{not:Fv}, the weight filtration $\mathfrak{F}_\mu$ given by $\mathfrak{n}^{(0)}=\mathfrak{n}^{(1)}=\mathfrak{n}$,
and for $i\ge1$,
\[
\mathfrak{n}^{(i+1)}=\left[\mathfrak{n},\mathfrak{n}^{(i)}\right]+\left[v_{\mu},\mathfrak{n}^{(i-1)}\right].
\]
By Proposition~\ref{prop:nested}, the filtration depends only on the class of $v_{\mu}$ in $\mathfrak{n}/[\mathfrak{n},\mathfrak{n}]$. When the random walk is $\mathfrak{n}$-centred, this is simply the lower central series.

Our goal is the following theorem.

\begin{theorem}\label{thm:norm} Let $\mu$ be a compactly supported measure on $G=N \rtimes Q$. Denote $w_n$ the random walk at time $n$. Then the family
\[
\left( \frac{1}{\sqrt{n}}\max_{k\le n}\left(\left|w_k\exp\left(-kv_{\mu}\right)\right|_{\mathfrak{F}_\mu}\right) \right)_{n\in \mathbb{N}}
\]
has subgaussian concentration, i.e. there exists $c_1,c_2>0$ such that
\[
\sup_{n\in \mathbb{N}} \mathbb{P} \left( \frac{1}{\sqrt{n}}\max_{k\le n}\left(\left|w_k\exp\left(-kv_{\mu}\right)\right|_{\mathfrak{F}_\mu}\right) \ge t\right) \le c_2e^{-c_1t^2}.
\]
The constants $c_1,c_2$ depend only on the Lie group via its dimension $d=\mathrm{dim}(N)$ and degree of nilpotency $s$, and on the measure $\mu$ via  the radius $R_\mu$, and the spectral constant $\kappa_\mu$ .
\end{theorem}

We stress the fact that $\mu$ is not assumed to be non-degenerate. Note that the norm $\left|\cdot\right|_{\mathfrak{F}_\mu}$
depends on the weight filtration $\mathfrak{F}_\mu$, which is determined by the drift $v_{\mu}$. With a slight abuse of notation, we still denote  $|\cdot|_{\mathfrak{F}_\mu}$ the norm on $\mathfrak{n}$ is pushed forward to $G$ by setting $|(\xi,\kappa)|_{\mathfrak{F}_\mu}:=|\xi|_{\mathfrak{F}_\mu}$.

\begin{remark}\label{rk:balistic}
The dependance of the constant on the spectral constant $\kappa_\mu$ is necessary. For instance consider the group $\mathbb{R}\rtimes \mathbb{Z}/2\mathbb{Z}$ with measure $\mu_\varepsilon=(1-\varepsilon)\delta_{(1,0)}+\varepsilon\delta_{(0,1)}$. Then at time $n$ the random walk lies in $[-n,n]\times \mathbb{Z}/2\mathbb{Z}$ and takes value $(n,0)$ with probability $(1-\varepsilon)^n$. Thus for small enough parameter $\varepsilon>0$, we have
\[
\mathbb{P}_{\mu_\varepsilon}\left( \frac{1}{\sqrt{n}} |w_n|_{\mathfrak{F}_\mu} \ge \frac{\sqrt{n}}{2}\right) \ge \frac{1}{2},
\]
preventing the existence of uniform constants $c_1,c_2>0$.
\end{remark}

\begin{remark}\label{rk:algnorm}
One can also state the concentration inequality of Theorem~\ref{thm:norm} with the homogeneous norm $|\cdot|_{\mathfrak{F}_\mu,\|\cdot\|_e}$ associated with the original $Q$-invariant euclidean norm~$\|\cdot\|_e$. However this raises a further dependence on the bilinearity constant $C_{\|\cdot\|_e}$ because of the first point of Proposition~\ref{prop:normsubadd}. This dependence is necessary. For instance on the first Heisenberg group with $\mathfrak{n}=\langle e_1,e_2,e_3|[e_1,e_2]=e_3\rangle$ taking euclidean norms of the form $\sqrt{x_1^2+x_2^2+\alpha x_3^2}$ with $\alpha$ tending to infinity prevents from having a uniform bound.
\end{remark}

Let us now proceed to the proof of Theorem~\ref{thm:norm}. We first assume that the average $\mathbb{E}\pi_{\mathrm{ab}}^{\mathfrak{n}}\mu=:v_\mu$ in $\mathfrak{m}_1$ is $\mathrm{Ad}(Q)$-invariant. When $\mu$ is centred, the random variable $g_{1}\ldots g_{k}\exp\left(-kv_{\mu}\right)$ is a random walk as we have $v_\mu=0$ according to Lemma~\ref{lem:conjugate}, but in general, it is only a product of centred independent increments.

As the adjoint $Q$-action preserves $v_\mu$, we have 
\[
g_j \exp(-v_\mu)=\exp(\xi_j)\kappa_j\exp(-v_\mu)=\exp(\xi_j\ast(-\kappa_jv_\mu))\kappa_j=\exp(\xi_j\ast(-v_\mu))\kappa_j=\exp(\zeta_j)\kappa_j
\]
 by setting $\zeta_j:=\xi_j\ast(-v_\mu)$. The Baker-Campbell-Hausdorff formula ensures that $\zeta_j=\xi_j-v_\mu \mod [\mathfrak{n},\mathfrak{n}]$, whence is centred. The sequence $(\zeta_j)$ is i.i.d. with compactly supported distribution.

Recall the notation $q_j=\kappa_1\dots\kappa_j$. We have 
\begin{align*}
g_{1}\ldots g_{n}\exp\left(-nv_{\mu}\right) & = \prod_{j=1}^n \exp\left((j-1)v_\mu\right)g_j \exp\left(-jv_\mu\right) \\
& = \prod_{j=1}^n \exp\left((j-1)v_\mu\right)\exp(\zeta_j)\kappa_j \exp\left(-(j-1)v_\mu\right) \\
& = \left( \prod_{j=1}^n \exp\left((j-1)v_\mu\right)\exp(q_{j-1}\zeta_j) \exp\left(-(j-1)v_\mu\right) \right) q_n.
\end{align*}
Thus if we denote $g_{1}\ldots g_{n}\exp\left(-nv_{\mu}\right)=(y_n,q_n)$ with $y_n=z_n-nv_\mu \in \mathfrak{n}$, we have :
\begin{align}\label{eq:increment}
y_n=y_{n-1}\ast (n-1)v_\mu \ast q_{n-1}\zeta_n\ast \left( -(n-1)v_\mu\right),
\end{align}
for some centred i.i.d. compactly supported random variables $(\zeta_k)$.

For $v \in \R v_\mu$ and $\zeta\in \mathfrak{n}$, using the BCH formula and that $[v,\zeta]\in \mathfrak{n}^{(3)}$, we obtain that
\begin{align}\label{eq:modn4}
v\ast\zeta\ast(-v) = \zeta+[v,\zeta] \mod \mathfrak{n}^{(4)}.
\end{align}

\begin{remark}\label{rk:power2}
It is sufficient to prove Theorem~\ref{thm:norm} in the case $n$ is a power of $2$. Indeed, if $\frac{1}{\sqrt{2n}}\max_{k\le 2n}X_k$ is subgaussian with constants $c_1,c_2$, then for any $n\le p \le 2n$ the variable $\frac{1}{\sqrt{p}}\max_{k\le p}X_k$ is subgaussian with
\[
\mathbb{P}\left(\frac{1}{\sqrt{p}}\max_{k\le p}X_k \ge s\right)\le 
\mathbb{P}\left(\frac{1}{\sqrt{2n}}\max_{k\le 2n}X_k \ge \frac{s}{\sqrt{2}}\right) \le c_2e^{-\frac{c_1}{2}s^2}.
\]
\end{remark}

According to this remark, it will be useful to notice that
\begin{align*}
(y_{2n},q_{2n}) & =g_1\dots g_n \exp(-nv_\mu)\exp(nv_\mu)g_{n+1}\dots g_{2n}\exp(-2nv_\mu) \\
&= (y_n,q_n)\exp(nv_\mu)(y'_n,q'_n)\exp(-nv_\mu)
\end{align*}
where $(y'_n,q'_n)$ is an independent sample. We get
\begin{align}\label{eq:y2n}
(y_{2n},q_{2n})=(y_n \ast nv_\mu\ast q_ny'_n,\ast(-nv\mu), q_nq'_n).
\end{align}

\subsection{Proof of concentration in adapted norm}

Denote $\theta_i:\mathfrak{n}\to \mathfrak{m}_i$ the orthogonal projection, giving the decomposition $y=\theta_1(y)+\dots+\theta_c(y)$ for any $y \in \mathfrak{n}$.
By definition of the norm $|\cdot|_{\mathfrak{F}_\mu}$, see (\ref{def:norm}), it is sufficient to show that for each $c\ge 1$ the family
\begin{align}\label{eq:thetac}
\left( \frac{1}{\sqrt{n}}\max_{k\le n}\|\theta_c(y_k)\|^{\frac{1}{c}} \right)_{n\in \mathbb{N}} \quad  \textrm{ is subgaussian.}
\end{align}

We proceed by induction on $c$.
The induction starts in the abelianization of $\mathfrak{n}$. In this case, (\ref{eq:increment}) immediately gives
\[
\theta_1(y_n)=\sum_{k=1}^n q_{k-1}\theta_1(\zeta_k)
\]
where $(\zeta_k)$ are compactly supported i.i.d. centred. Then $\theta_1(y_n)$ is a martingale with bounded increments. We get (\ref{eq:thetac}) for $c=1$ by the Azuma-Hoeffding inequality recalled in Proposition~\ref{prop:AH}. The constants involved depend only on $R_\mu$ and $d$.


We now prove the case $c=2$.
 
 \begin{lemma}\label{lem:theta2}
 The family
 \[
\left(  \frac{1}{\sqrt{n}} \max_{k\le n} \|\theta_2(y_k)\|^{\frac{1}{2}} \right)_{n \in \mathbb{N}}
 \]
 has subgaussian concentration, with constants depending only on $d,R_\mu$, and $\kappa_\mu$.
 \end{lemma}


\begin{proof}[Proof of Lemma~\ref{lem:theta2}]
By induction for $c=1$ we assume that 
\begin{align}\label{eq:theta1}
\left( \frac{1}{\sqrt{n}}\max_{k\le n}\|\theta_1(y_k)\| \right)_n \quad  \textrm{ is subgaussian.}
\end{align}

Equivalently to the lemma, we prove that the family $\left(\frac{1}{n}\max_{k\le n}\|\theta_2(y_n)\|\right)_n$ is subexponential. According to Proposition~\ref{prop:SubexpSubGauss}, we have to show that there exist $E,K,b>0$ such that for all unit vector $u \in \mathfrak{m}^{(2)}$, 
\begin{align}\label{eq:scalu}
\left(\frac{1}{n} \max_{k \le n} \langle \theta_2(y_k) |u \rangle \right)_n \quad \textrm{ is }E,K,b\textrm{-subexponential.}
\end{align}

First (\ref{eq:increment}), (\ref{eq:modn4}) and BCH formula give
\[
\theta_2(y_n)=\theta_2(y_{n-1})+\theta_2(q_{n-1}\zeta_n)+\frac{1}{2}\left[\theta_1(y_{n-1}),\theta_1(q_{n-1}\zeta_n)\right]
\]
by induction.
As $\zeta_n$ is centred, we have $\mathbb{E}\left[\theta_1(y_{n-1}),\theta_1(q_{n-1}\zeta_n)\right]=0$. It follows that 
\[
\mathbb{E}\theta_2(y_n)=\mathbb{E}\theta_2(y_{n-1})+\mathbb{E}\theta_2(q_{n-1}\zeta_n)=\sum_{k=1}^n \mathbb{E}\theta_2(q_{k-1}\zeta_k).
\]
We get $\frac{1}{n}\mathbb{E}\theta_2(y_n) \le E=\sup_{q,\zeta}\|\theta_2(q\zeta)\|=\sup_\zeta\|\theta_2(\zeta)\|$. Here the supremum is taken over $q\in Q$ and $\zeta \in \mathrm{supp}(\pi^\mathfrak{n}\mu)-v_\mu$. It does not depend on $Q$ as its action preserves the norm and the filtration.  
The set $\mathrm{supp}(\pi^\mathfrak{n}\mu)-v_\mu$ is compact, bounded in terms of $R_\mu$ and~$\kappa_\mu$.

Given a vector valued random variable $Z$, we denote $Z^0=Z-\E(Z)$ the recentred random variable. There remains to show the existence of $K,b>0$ such that $\left(\frac{1}{n}\theta_2(y_n)^0\right)_n$ is $(K,b)$-subexponential. According to Remark~\ref{rk:power2}, it is sufficient to do it for $n=2^k$ which we do by induction on $k$. 

Now by (\ref{eq:y2n}), (\ref{eq:modn4}) and the BCH formula, we have
\begin{align*}
\theta_2(y_{2n})=\theta_2(y_n)+\theta_2(q_ny'_n)+\frac{1}{2}\left[ \theta_1(y_n),\theta_1(q_n y'_n)\right].
\end{align*}
As $\theta_1(y_n)$ is centred, we get
\begin{align}\label{eq:th2y}
\theta_2(y_{2n})^0=\theta_2(y_n)^0+\theta_2(q_ny'_n)^0+\frac{1}{2}\left[ \theta_1(y_n),\theta_1(q_n y'_n)\right].
\end{align}

As we chose a norm with bilinearity constant $C_{\|\cdot\|} \le 1$, we have
\[
\left\|\frac{1}{2}\left[ \theta_1(y_n),\theta_1(q_n y'_n)\right]\right\| \le \frac{1}{2} \left\| \theta_1(y_n)\right\| . \left\| \theta_1(q_ny'_n)\right\|.
\]

By assumption (\ref{eq:theta1}), the sequence $\frac{1}{\sqrt{n}}\theta_1(y_n)$, and hence the sequence $\frac{1}{\sqrt{n}}\theta_1(q_ny'_n)$ as well, is subgaussian, so their product is subexponential. By Proposition~\ref{prop:products}, this gives $K_1,b_1>0$, depending on $d,R_\mu$, such that
\begin{align}\label{eq:K1b1}
\left( \frac{1}{2n}\left[ \theta_1(y_n),\theta_1(q_n y'_n)\right] \right)_n \textrm{ is }(K_1,b_1)-\textrm{subexponential.}
\end{align}

Consider $K\ge K_1/4(1-\frac{1}{\sqrt{2}})^2$ and $b < b_1 2(1-\frac{1}{\sqrt{2}})$. We prove by induction that when $n=2^k$, the random variable $\frac{1}{n} \theta_2(y_n)^0$ is $(K,b)$-subexponential.
(Initiation is done by choosing also $K$ large and $b$ small enough that it holds for $\theta_2(\zeta)$, which has finite support, depending on $R_\mu$.)

Assume now that this is true for $n$ and consider (\ref{eq:th2y}). To ease notations denote $\alpha_n:=\frac{1}{n}\theta_2(y_n)^0$ and $\beta_n:= \frac{1}{n}\theta_2(q_ny'_n)^0$. We have for all $|t|\le b$ and all unit $u \in \mathfrak{m}_{2}$:
\[
\mathbb{E}\exp\left(t \langle \alpha_n ,u\rangle\right) \le \exp(Kt^2)
\]
and also, as the $Q$ adjoint action on $\mathfrak{n}$ preserves the norm, conditioning by the values of the $n$ first increments $g_1,\dots,g_n$, we have
\[
\mathbb{E}\left[\exp\left(t \langle \beta_n ,u\rangle\right)|g_1,\dots,g_n\right]=\mathbb{E}\left[\exp\left(t\left\langle\frac{1}{n}\theta_2(y'_n)^0,q_n u\right\rangle\right)|g_1,\dots,g_n\right] \le \exp(Kt^2).
\]

Our key observation is that being ``essentially'' independent, $\alpha_n$ and $\beta_n$ are such that 
$\frac{1}{\sqrt{2}}(\alpha_n+\beta_n)$ is $(K,b)$-subexponential. 
Let us check this point. For all $|t|\leq b$
\begin{eqnarray*}
\E\left[\exp\left(t\left\langle\frac{\alpha_n+\beta_n}{\sqrt{2}},u\right\rangle\right) \mid g_1,\dots,g_n\right]& = & \exp\left(t\left\langle\frac{\alpha_{n}}{\sqrt{2}},u\right\rangle\right) \E\left[\exp\left(t\left\langle\frac{\beta_{n}}{\sqrt{2}},u\right\rangle\right) \mid g_1,\dots,g_n\right]\\
                                                            & \leq &  \exp\left(t\left\langle\frac{1}{\sqrt{2}}\alpha_{n},u\right\rangle\right)\exp(Kt^2/2)\\
\end{eqnarray*}
Hence we deduce that 
\[\E\exp\left(t\left\langle\frac{\alpha_n+\beta_n}{\sqrt{2}},u\right\rangle\right)\leq \exp(Kt^2/2)\exp(Kt^2/2)=\exp(Kt^2).\]

We are going to exploit the fact that 
\[
\frac{1}{2}(\alpha_n+\beta_n)=\frac{1}{\sqrt{2}}\left(\frac{1}{\sqrt{2}}(\alpha_n+\beta_n)\right).
\]
To lighten notation, let $\eta=\frac{1}{\sqrt{2}}$. Equation (\ref{eq:th2y}) can be rewritten as
 \[
 \frac{1}{2n}\theta_2(y_{2n})^0= \eta\left(\frac{1}{\sqrt{2}}(\alpha_n+\beta_n)\right)+(1-\eta)\left(\frac{1}{4(1-\eta)n}\left[\theta_1({y}_{n}),\theta_1(q_{n}y'_n)\right]\right).
 \]
 
By (\ref{eq:K1b1}) and Remark~\ref{rk:lambda}, we have that $\frac{1}{4(1-\eta)n}\left[\theta_1({y}_{n}),\theta_1(q_{n}y'_n)\right]$ is $(\frac{K_1}{4(1-\eta)^2},2(1-\eta)b_1)$-subexponential. The choice of $K,b$ and Corollary \ref{cor:convexSubExp} imply that $\frac{1}{2n}\theta_2(y_{2n})^0$ is $(K,b)$-subexponential, finishing the induction.

We have proved that $\left(\frac{1}{2^k}\theta_2(y_{2^k})^0\right)_n$ is subexponential. As $\theta_2(y_n)^0$ is a martingale, Proposition~\ref{prop:MaxConcentrated} ensures that $\left(\frac{1}{2^k}\max_{j\le 2^k}\theta_2(y_j)^0\right)_n$ is subexponential. This proves the lemma for $n$ a power of $2$, which is sufficient by Remark~\ref{rk:power2}. Proposition~\ref{prop:MaxConcentrated} and Remark~\ref{rk:power2} modify the constants of concentration only by universal multiplicative factors.
\end{proof}

There remains to treat the case $c\ge 3$.


\begin{lemma}\label{lem:theta3y} For each $c\ge 1$, the family
\[
\left(\frac{1}{\sqrt{n}} \max_{k \le n}\|\theta_c(y_k)\|^{\frac{1}{c}}\right)_n
\]
is subgaussian, with constants depending only on $d, R_\mu, \kappa_\mu$.
\end{lemma}

\begin{proof} 
We assume by induction that the lemma holds up to $c-1$. In particular, we assume that for all $i \le c-1$, the family $\left(\frac{1}{\sqrt{n}} \max_{k \le n}\|\theta_i(y_k)\|^{\frac{1}{i}}\right)_n$ is subgaussian.

Similarly to (\ref{eq:y2n}), we have $y_{k+l}=y_k \ast kv_\mu\ast  q_k y_l \ast(-kv_\mu)$ where $y_k$ and $y_l$ denote independent samples at times $k$ and $l$.

The BCH formula gives that
\begin{align}\label{eq:ykl}
\theta_c(y_{k+l})=\theta_c(y_k)+\theta_{c}(q_k y_l) + z_{k,l},
\end{align}
where $z_{k,l}$ is formaly a finite linear combination (with explicit coefficients) of commutators of the form $[u_1,\dots,[u_{m-1},u_m]]$ with $m\ge 2$ where the factors $u_j$ are either random variables $\theta_{i_j}(y_k)$ or $\theta_{i_j}(q_ky_l)$, or are deterministic taking value $ k v_\mu$. 

For instance, for $c=3$ we can compute explicitly using (\ref{eq:modn4})
\[
z_{k,l}=[kv_\mu, \theta_1(q_k y_l)])+\frac{1}{2}[\theta_1(y_{k}),\theta_2(q_{k}y_l)]+\frac{1}{2}[\theta_2(y_{k}),\theta_1(q_{k}y_l)].
\]

For each iterated commutator of the form $[u_1,\dots,[u_{m-1},u_m]]$, we define
\[
i_j :=\left\{ \begin{array}{ll} 
i & \textrm{ if } u_j \in \{\theta_i(y_k),\theta_i(q_ky_l)\} \\
2 & \textrm{ if } u_j =kv_\mu
\end{array}\right.
\]
By Lemma~\ref{lem:commute}, the commutator $[u_1,\dots,[u_{m-1},u_m]]$ belongs to $\mathfrak{m}_{c}$ if and only if $\sum_{j=1}^m i_j=c$. In particular, we  have $i_j \le c-1$ for all $j$. For such a given iterated commutator, set $J:=\{1\le j \le m: u_j=kv_\mu\}$.

We first claim that $\left(\frac{1}{\sqrt{n}}\max_{k,l\le n}\|z_{k,l}\|^{1/c}\right)$ is subgaussian, or equivalently that the family 
\[
\left(\frac{1}{n^{c/2}}\max_{k,l\le n}\|z_{k,l}\|\right)\quad \textrm{ is }\frac{2}{c}\textrm{-concentrated.}
\]
By the above, we have that $\max_{k,l\le n}\|z_{k,l}\|$ is bounded above by a finite sum with explicit coefficients of iterated commutators $[u_1,\dots,[u_{m-1},u_m]]$ with factors satisfying $\|u_j\|\le n\|v_\mu\|$ for $j \in J$, and for $j\notin J$:
\[
\|u_j\| \le \max \left(\max_{k\le n}\|\theta_{i_j}(y_k)\|, \max_{k,l\le n}\|\theta_{i_j}(q_ky_l)\|  \right).
\]

By induction, we have that
$\left(\frac{1}{\sqrt{n}} \max_{k \le n}\|\theta_i(y_k)\|^{\frac{1}{i}}\right)$ is subgaussian, or equivalently that 
\[
\left(\frac{1}{n^{i/2}}\max_{k \le n} \|\theta_i(y_k)\| \right) \quad \textrm{ is }\frac{2}{i} \textrm{-concentrated}
\]
and for the same reasons
\[
\left(\frac{1}{n^{i/2}}\max_{k,l \le n} \|\theta_i(q_ky_l)\|\right) \quad \textrm{ is }\frac{2}{i} \textrm{-concentrated.}
\]
It follows that $\frac{u_j}{n^{i_j/2}}$ is $\frac{2}{i_j}$-concentrated.

By Lemma~\ref{lem:multicomm} with $C_{\|\cdot\|}\le 1$, we have
\[
\frac{1}{n^{c/2}}\|[u_1,\dots,[u_{m-1},u_m]]\| \le \prod_{j=1}^m \frac{\|u_j\|}{n^{i_j/2}} \le \|v_\mu\|^{\#J}\prod_{j \notin J} \frac{\|u_j\|}{n^{i_j/2}}.
\]
By Corollary~\ref{cor:2/c}, the left-hand side is $2/(\sum_{j \notin J}i_j)$-concentrated. A fortiori, it is $2/c$-concentrated. We have proved that $\frac{1}{\sqrt{n}}\|[u_1,\dots,[u_{m-1},u_m]]\|^{1/c}$ is subgaussian. By finite sum, we get our claim that $\left(\frac{1}{\sqrt{n}}\max_{k,l\le n}\|z_{k,l}\|^{1/c}\right)$ is subgaussian. The constants involved depend only on the previous step of induction. By Proposition~\ref{prop:p-ConcentrationInequalities}, there is $C_1$ such that
\begin{align}\label{eq:Cpz}
\forall p \ge 1, \quad \left\|\frac{1}{\sqrt{n}}\max_{k,l\le n}\|z_{k,l}\|^{1/c}\right\|_p \le C_1 p^{\frac{1}{2}}.
\end{align}

We are now ready to prove the lemma for $n=2^k$ by induction on $k$, which is sufficient by Remark~\ref{rk:power2}. There remains only to prove that for some $C$ and for all $n$ power of $2$
\begin{align}\label{eq:Cpthc}
\forall p \ge 1, \quad \left\|\frac{1}{\sqrt{n}}\max_{k\le n}\|\theta_c(y_k)\|^{1/c}\right\|_p \le C p^{\frac{1}{2}}
\end{align}
We let $C \ge C_1/(2^{1/2}-2^{1/c})$. We also assume that $C$ is large enough for initialisation at $n=1$, which depends only on $R_\mu$.
Equation (\ref{eq:ykl}) and subadditivity give
\[
\max_{j\le 2n} \|\theta_c(y_j)\|^{\frac{1}{c}} \le 2^{\frac{1}{c}}\max_{k\le n}\|\theta_c(y_k)\|^{\frac{1}{c}}+\max_{k,l\le n}\|z_{k,l}\|^{\frac{1}{c}}
\]

Taking $p$-norms and dividing by $\sqrt{2n}$, we get
\begin{align*}
\left\| \frac{1}{\sqrt{2n}}\max_{j\le 2n}\|\theta_c(y_j)\|^{\frac{1}{c}}\right\|_p & \le 2^{\frac{1}{c}-\frac{1}{2}}\left\| \frac{1}{\sqrt{n}}\max_{k\le n}\|\theta_c(y_k)\|^{\frac{1}{c}}\right\|_p+2^{-\frac{1}{2}}\left\|\frac{1}{\sqrt{n}}\max_{k,l\le n}\|z_{k,l}\|^{\frac{1}{c}}\right\|_p 
\\
& \le \left( 2^{\frac{1}{c}-\frac{1}{2}}C+2^{-\frac{1}{2}}C_1\right)p^{\frac{1}{2}} 
\end{align*}
by induction hypothesis and (\ref{eq:Cpz}). The choice of $C$ gives (\ref{eq:Cpthc}).
\end{proof}

\begin{proof}[Proof of Theorem~\ref{thm:norm}]
We have $\left|w_n\exp\left(-nv_{\mu}\right)\right|_{\mathfrak{F}_\mu}=|y_n|_{\mathfrak{F}_\mu}=\sup_{1\le r \le c} \|\theta_r(y_n)\|^{\frac{1}{r}}$ for some $c\le 2s$.  By (\ref{eq:theta1}), Lemma~\ref{lem:theta2} and Lemma~\ref{lem:theta3y}, this is a supremum of finitely many subgaussian random variables, whence is subgaussian. The constants  also depend on the number of terms, which is at most $2s$. 

At this stage we have proved Theorem~\ref{thm:norm} under the assumption that $\mathbb{E}\pi_{\mathrm{ab}}^{\mathfrak{n}}\mu=:v_\mu$ in $\mathfrak{m}_1$ is $\mathrm{Ad}(Q)$-invariant. In general, we further use Lemma~\ref{lem:conjugate}, which provides $y\in \mathfrak{n}$ of norm at most $R_\mu/\kappa_\mu$ such that $c_y\mu$ satisfies the assumptions. As conjugating by $(y,1)$ only modifies the distances by at most an additive constant bounded by $2R_\mu/\kappa_\mu$, the theorem is valid in full generality. (The same reasoning applies to show that Lemma~\ref{lem:theta2} and~\ref{lem:theta3y} are also valid in full generality.)
\end{proof}

\section{Concentration in word norms}

\subsection{Structure of groups of polynomial growth}

Let $G$ be a compactly generated locally compact group. We assume $G$ has polynomial growth: for some (equivalently any) compact generating neighbourhood $U$ of the identity, there are constants $C,d>0$ such that $\mathrm{Haar}(U^r) \le Cr^d$ for all positive integer~$r$.

Such a group $G$ admits a maximal compact normal subgroup $K$ by Losert~\cite[Prop. 1]{LosII}. Moreover, there exists a connected, simply connected, nilpotent Lie group $N$ and a compact group $Q$ acting faithfully on $N$, together with a homomorphism $\varphi:G/K \hookrightarrow N\rtimes Q$ which is injective and has cocompact closed image, by Losert~\cite[Thm~2]{Losert}. So we have a sequence of maps
\begin{align}\label{eq:maps}
K \hookrightarrow G \overset{\pi}{\twoheadrightarrow} G/K \overset{\varphi}{\hookrightarrow} N \rtimes Q.
\end{align}
Note that $N$ is $s$-nilpotent for some $s\le \sqrt{d}$.
We consider symmetric compact generating neighbourhoods of identity $U$ in $G$, $S$ in $G/K$,  $T$ in $N \rtimes Q$ and $R$ in $N$, and their associated word norms denoted $|\cdot|_U$ in $G$ and similarly for the others. We assume moreover that $S=\pi(U)$, and that $T$ satisfies three conditions : firstly $S \subset T$ (we view $\varphi$ as an embedding and identify $S$ with its image $\varphi(S)$), and secondly $\bar{\Omega} \subset T$ where $\bar{\Omega}\subset N\rtimes Q$ is a compact subset  such that $N\rtimes Q=\mathrm{Im}(\varphi)\bar{\Omega}$. So far this is satisfied by $T_1=S^{\pm1}\cup\bar{\Omega}^{\pm 1}$. We assume thirdly that $T=QRQ$, which can be done by choosing $R\supset \pi_N(T_1)$ for the natural projection $\pi_N: N\rtimes Q \to N$.

\begin{lemma}\label{lem:metricsG}
In the setting above, we have
\begin{enumerate}
\item for all $g \in G$, $|\pi(g)|_S \le |g|_U \le |\pi(g)|_S+\mathrm{diam}_U(K)$,
\item there exists a constant $c_{S,T}$  such that $T^3 \cap \mathrm{Im}(\varphi) \subset S^{c_{S,T}}$, and for all $h \in \mathrm{Im}(\varphi)$, $|h|_T \le |h|_S \le c_{S,T}|h|_T$, 
\item for all $h \in N \rtimes Q$, $|h|_T \le |\pi_N(h)|_R+1.$
\end{enumerate}
Thus $|g|_U \le c_{S,T}|\pi_N\circ\pi(g)|_R+c_{S,T}+\mathrm{diam}_U(K)$.
\end{lemma}

This is well-known but we provide a proof for completeness.

\begin{proof}
The left hand side inequalities are trivial. Assume $|\pi(g)|_S=k$. It gives $s_1\dots s_k=\pi(g)$ for some $s_i=\pi(u_i) \in S=\pi(U)$, so $(u_1\dots u_k)^{-1} g \in K$, proving the first point. For the second point, $T^3$ is compact and $\mathrm{Im}(\varphi)$ closed so the required constant $c=c_{S,T}$ exists. Assume $|h|_T=k$. It gives $h=t_1\dots t_k$ for $t_i\in T$. For each $0\le j \le n$, there is $h_j \in \mathrm{Im}(\varphi)$ such that $t_1\dots t_j \in h_j \bar{\Omega}$. Then $h_j^{-1}h_{j+1} \in \bar{\Omega}t_{j+1}\bar{\Omega}^{-1} \subset T^3$, so $h_j^{-1}h_{j+1} \in S^{c}$. It follows that $h=h_k=\prod(h_j^{-1}h_{j+1}) \in S^{ck}$. In the third point, the term $1$ takes into account the case $h \in Q$. 
\end{proof}

\subsection{Concentration for random walks on groups of polynomial growth}

Let $\tilde{\mu}$ be a compactly supported probability measure on $G$. We let $\mu=\varphi \circ \pi (\tilde{\mu})$ be the associated probability measure on $N\rtimes Q$ via the maps~(\ref{eq:maps}). We denote by $\omega_n$ the time $n$ sample of the random walk of law $\tilde{\mu}$. As in the previous section, $\varphi\circ\pi(\omega_n)=w_n$ denotes the $\mu$-random walk trajectory.

\begin{theorem}\label{thm:wordconcentration}
Let $G$ be a compactly generated locally compact group of polynomial growth, with a compactly supported probability measure $\tilde{\mu}$.
\begin{enumerate}
\item If $\tilde{\mu}$ is centred, then $\left(\frac{1}{\sqrt{n}}\max_{k\le n} |\omega_k|_U \right)_n$ is subgaussian.
\item In general, if $N$ is $s$-nilpotent, there exists a trajectory $(g_n)_{n}$ in $G$ depending only on $\tilde{\mu}$ such that 
\[
\left( \frac{1}{n^{\frac{2s-1}{2s}}} \max_{k\le n}|\omega_kg_k^{-1}|_U \right)_n \quad \textrm{is }\frac{2s}{2s-1}\textrm{-concentrated.}
\]
\end{enumerate}
\end{theorem} 

The trajectory $(g_n)$ is a sequence of points in $G$ whose images in $N\rtimes Q$ are at bounded distance of  $\exp(nv_\mu)$.

The parameter $s$ appears both in the exponent of concentration and the appropriate scaling. We note that the concentration is always better than subexponential, but the scaling, though always sublinear, does not admit an exponent uniform in $s$.

The constants involved in Theorem~\ref{thm:wordconcentration} depend on compactness of $K$ and cocompactness of the image of $\varphi$, as seen from (\ref{eq:metricsUR}) and (\ref{eq:Ccompact}) below. 

Let us denote $\varphi\circ\pi(\omega_n)=\exp(z_n)q_n$ the image of $\omega_n$ in $N\rtimes Q$.
Lemma~\ref{lem:metricsG} gives
\begin{align}\label{eq:metricsUR}
|\omega_n|_U \le C (|\exp(z_n)|_R+1) \quad \textrm{with }C=\max(c_{S,T},\mathrm{diam}_U(K)).
\end{align}
By Proposition~\ref{prop:normsubadd}, we get a constant $C_2$ such that, we get
\begin{align}\label{eq:normsUL}
|\omega_n|_U \le C(C_2+C_2|z_n|_{\mathcal{L}}) \le C'(|z_n|_{\mathcal{L}}+1)
\end{align}
for some constant $C'$.

\begin{proof}[Proof of part 1. of Theorem~\ref{thm:wordconcentration}]
If $\tilde{\mu}$ is centred, then so is $\mu$. We apply Theorem~\ref{thm:norm} with $v_\mu=0$. In particular, $\mathcal{F}_{v_\mu}=\mathcal{L}$ coincides with the lower central series and $y_n=z_n-nv_\mu=z_n$.  By Lemma~\ref{lem:theta3y}, we have
\[
\mathbb{P}\left(\frac{1}{\sqrt{n}}\max_{k\le n} |\omega_k|_U\ge t\right) \le \mathbb{P}\left(\frac{1}{\sqrt{n}}\max_{k\le n} |y_k|_{\mathcal{L}}\ge \frac{t-1}{C'}\right)\le c_2\exp\left(-\frac{c_1}{C'^2} (t-1)^2\right)
\]
\end{proof}

In general, the filtration $\mathcal{F}_{v_\mu}$  does not coincide with the lower central series. We take into account the distorsion for the word metric by using Proposition~\ref{prop:nested}. Recall that we have two filtrations of $\mathfrak{n}$ denoted $\mathcal{F}_{\mu}=(\mathfrak{n}^{(i)})_{i=1}^c$ and $\mathcal{L}=(\gamma_j(\mathfrak{n}))_{j=1}^s$. We get two orthogonal decompositions of the Lie algebra
\[
\mathfrak{n}=\mathfrak{m}_{1}\oplus\dots\oplus\mathfrak{m}_{c}=\mathfrak{m}^{1}\oplus\dots\oplus\mathfrak{m}^{s},
\]
where $\mathfrak{m}_{i}\simeq \mathfrak{n}^{(i)}/\mathfrak{n}^{(i+1)}$ is the orthogonal complement of $\mathfrak{n}^{(i+1)}$ in $\mathfrak{n}^{(i)}$, and $\mathfrak{m}^{j}\simeq \gamma_j(\mathfrak{n})/\gamma_{j+1}(\mathfrak{n})$ is the orthogonal complement of $\gamma_{j+1}(\mathfrak{n})$ in $\gamma_j(\mathfrak{n})$. We denote the associated decompositions of $y\in \mathfrak{n}$ by
\[
y=\theta_1(y)+\dots+\theta_c(y)=\sigma_1(y)+\dots+\sigma_s(y).
\]
Proposition~\ref{prop:nested} gives $\mathfrak{n}^{(i)}\subset \gamma_{\lfloor\frac{i}{2}\rfloor+1}(\mathfrak{n})$. Therefore 
\[
y=\sum_{i=1}^c\sum_{j=\lfloor\frac{i}{2}\rfloor+1}^s \sigma_j(\theta_i(y))=\sum_{j=1}^s \sum_{i=1}^{2j-1} \sigma_j(\theta_i(y)).
\]
In particular, $\sigma_j(y)=\sum_{i=1}^{2j-1} \sigma_j(\theta_i(y))$ for any $y\in \mathfrak{n}$.

\begin{proof}[Proof of part 2. of Theorem~\ref{thm:wordconcentration}]
The embedding $\varphi:G/K  \hookrightarrow N \rtimes Q$ is cocompact, for all $n$ there exists $a_n \in \mathrm{Im}(\varphi)$ with $d_T(a_n,\exp(nv_\mu))\le \mathrm{codiam}_T(\mathrm{Im}(\varphi))$. Take $g_n$ in $G$ such that $\pi \circ \varphi(g_n)=a_n$.  Set 
\begin{align}\label{eq:Ccompact}
C=\max\left(c_{S,T},\mathrm{diam}_U(K),\mathrm{codiam}_T(\mathrm{Im}(\varphi))\right).
\end{align} 
By Lemma~\ref{lem:metricsG}, we have
\begin{align*}
|\omega_ng_n^{-1}|_U & \le C (|\varphi\circ \pi(\omega_ng_n^{-1})|_T+C) =C\left(|\exp(z_n)q_na_n^{-1}|_T+C \right)\\
 & \le C\left(|\exp(z_n-nv_\mu)q_n|_T+2C \right) \le C\left(|\exp(y_n)|_R+1+2C\right).
\end{align*}
By Proposition~\ref{prop:normsubadd}, we get 
\begin{align}\label{eq:UvsL}
|\omega_ng_n^{-1}|_U \le C(|y_n|_{\mathcal{L}}+1+2C).
\end{align}

On the other hand, Theorem~\ref{thm:norm}, or more directly Lemma~\ref{lem:theta3y}, implies that for all $1 \le i \le c$, the family $\left( \frac{1}{\sqrt{n}}\max_{k\le n}\|\theta_i(y_k)\|^{\frac{1}{i}}\right)$ is subgaussian. This is equivalent to  $\left( \frac{1}{n^{i/2}}\max_{k\le n}\|\theta_i(y_k)\|\right)$ being $\frac{2}{i}$-concentrated, which implies that for any $j$, the family $\left( \frac{1}{n^{i/2}}\max_{k\le n}\|\sigma_j(\theta_i(y_k))\|\right)$  is $\frac{2}{i}$-concentrated. We obtain that for all  $i \le 2j-1$,
\[
\left( \frac{1}{n^{(2j-1)/2}}\max_{k\le n}\|\sigma_j(\theta_i(y_k))\|\right) \quad \textrm{is }\frac{2}{2j-1}\textrm{-concentrated,}
\]
or equivalently that its $1/j$th power is $2j/(2j-1)$-concentrated. It follows that
\[
\frac{1}{n^\frac{2s-1}{2s}}\max_{k\le n}|y_k|_{\mathcal{L}} =\frac{1}{n^\frac{2s-1}{2s}}\max_{k\le n}\sup_{1\le j\le s}\|\sigma_j(y_k)\|^{\frac{1}{j}}\le \sup_{1\le j\le s} \frac{1}{n^{\frac{2j-1}{2j}}}\left( \max_{k\le n} \|\sigma_j(y_k)\|\right)^{\frac{1}{j}}
\]
is $2s/(2s-1)$-concentrated. The theorem follows using the comparison (\ref{eq:UvsL}).
\end{proof}

\subsection{Extensions by almost ultrametric subgroups}

For centred random walks, we still obtain concentration inequalities when the compact subgroup $K$ in the extension (\ref{eq:maps}) is replaced by a closed almost ultrametric subgroup. We consider a locally compact compactly generated group $G$ together with a closed subgroup $E$.

\begin{definition}\label{def:distortion}
A closed subgroup $E$ of $G$ is almost ultrametric if for some (or any) symmetric compact generating set $U$ of $G$, there exists a constant $C\ge 1$ such that
\[
|h_1\dots h_n|_U \le C\log(n)+C\max_{i}|h_i|_U
\]
for all $h_1,\dots,h_n \in E$.
\end{definition}

We first state a general result asserting that almost ultrametric subgroups do not modify concentration around the identity. We say a random walk $(\omega_k)_{k\in \mathbb{N}}$ of law $\tilde{\mu}$ on $G$ is \emph{ $\alpha$-concentrated at scale $f(n)$} if
\[
\mathbb{P}\left( \frac{1}{f(n)} \max_{k\le n}|\omega_k|_U\ge t\right) \le c_2\exp\left(-c_1t^\alpha\right)
\]
for some $c_1,c_2>0$. Denote $\pi:G \to G/E$ the quotient map and $\mu=\pi(\tilde{\mu})$.

\begin{proposition}\label{prop:EDconc}
Let $E$ be a closed almost ultrametric normal subgroup of $G$. If the  random walk for $\mu=\pi(\tilde{\mu})$ is centred and $\alpha$-concentrated at scale $f(n)$, then the $\tilde{\mu}$~random walk on $G$ is also $\alpha$-concentrated at scale $\max(f(n),\log(n))$.
\end{proposition}

This proposition immediately implies Proposition~\ref{prop:Edoesnotcount} as a particular case.
Together with Theorem~\ref{thm:wordconcentration}, it also implies the following corollary (Theorem~\ref{thmIntro:EbyP} in introduction).

\begin{corollary}\label{cor:EDmax}
Consider a locally compact compactly generated group $G$ in an exact sequence
\[
1\longrightarrow E \longrightarrow G \longrightarrow G/E \longrightarrow 1
\]
with $E$ almost ultrametric and $G/E$ of polynomial growth. Compactly supported centred random walks on $G$ are maximally diffusive.
\end{corollary}

\begin{proof}[Proof of Corollary~\ref{cor:EDmax}]
By Theorem~\ref{thm:wordconcentration}, compactly supported centred random walks on groups of polynomial growth are subgaussian at scale $\sqrt{n}$. Together with Proposition~\ref{prop:EDconc}, this implies the corollary.
\end{proof}

Proposition~\ref{prop:EDconc} relies on a purely geometric statement.

\begin{lemma}\label{lem:ED}
Let $E$ be almost ultrametric in $G$ and $U$ be a compact generating set of $G$. Denote $\pi:G\to G/E$. Let $u_1,\dots,u_n \in U$. 
\[
\textrm{If }\max_{k\le n}|\pi(u_1\dots u_k)|_{\pi(U)} \le m \quad \textrm{then} \quad \max_{k\le n}|u_1\dots u_k|_U\le C(\log(n)+4m).
\]
\end{lemma}

The constant $C$ comes from Definition~\ref{def:distortion}.

\begin{proof}[Proof of Proposition~\ref{prop:EDconc}] 
Lemma~\ref{lem:ED} ensures that if 
\[
\max_{k\le n}|\omega_k|_U \ge t 5C\max(f(n),\log(n))\ge C(\log(n)+4tf(n))
\]
then $\max_{k\le n} |\pi(\omega_k)|_{\pi(U)} \ge tf(n)$. It follows that for some $c_1,c_2>0$
\begin{align*}
\mathbb{P} \left( \max_{k\le n}|\omega_k|_U \ge t \max(f(n),\log(n))
\right) &\le 
\mathbb{P} \left( \max_{k\le n}|\pi(\omega_k)|_{\pi(U)} \ge \frac{tf(n)}{5C}
\right) \\
&\le c_2 \exp\left(-c_1\left(\frac{t}{5C}\right)^\alpha \right).
\end{align*}
\end{proof}

\begin{proof}[Proof of Lemma~\ref{lem:ED}]
First observe that 
\begin{align}\label{eq:eh}
\forall g \in G, \exists e \in E, h\in G, \textrm{ such that } g=eh \textrm{ and } |h|_U \le |\pi(g)|_{\pi(U)}.
\end{align}
Indeed, write $\pi(g)=\pi(u_1)\dots \pi(u_k)$ for $u_i\in U$ and $k=|\pi(g)|_{\pi(U)}$ and take $h=u_1\dots u_k$.

We claim that for any $g=u_1\dots u_n\in G$, there exists $e_1,\dots,e_n \in E$ and $h \in G$ such that $g=e_1\dots e_n h$ and
\begin{align}\label{eq:enh}
\max_{k\le n} |e_k|_U \le 2\max_{k\le n}|\pi(u_1\dots u_k)|_{\pi(U)}+\max_{k\le n} |u_k|_U \quad \textrm{and} \quad |h|_U \le |\pi(g)|_{\pi(U)}.
\end{align} 
We prove this claim by induction. The case $n=1$ is given by (\ref{eq:eh}). Given $g=u_1\dots u_n$, by induction we may assume $g=e_1\dots e_{n-1}h u_n$. Observation (\ref{eq:eh}) permits to write $hu_n=e_nh'$. To prove that $|e_n|_U$ is bounded above by the right hand side of (\ref{eq:enh}), it suffices to see that $e_n=h u_n h'^{-1}$ with $|h|_U \le |\pi(u_1\dots u_{n-1})|_{\pi(U)}$ and $|h'|_U \le |\pi(hu_n)|_{\pi(U)} \le |\pi(u_1\dots u_n)|_{\pi(U)}$. 

In particular when $u_1,\dots, u_n$ belong to $U$, this claim gives for any $j \le n$ that $u_1\dots u_j=e_1^j\dots e_j^jh_j$ with 
\[
\max_{i \le j} |e_i^j|_U \le 2\max_{i\le j} |\pi(u_1\dots u_i)|_{\pi(U)}+\max_{i\le j}|u_i|_U\le 2m+1.
\]
As $E$ is almost ultrametric, we deduce
\begin{align*}
|u_1\dots u_j|_U & \le C \log(j)+C\max_{i \le j}|e_i^j|_U+C|h_j|_U \\
& \le C\log(n)+3Cm+Cm \le C(4m+\log(n)),
\end{align*}
using $|h_j|_U \le |\pi(u_1,\dots u_j)|_{\pi(U)} \le m$.
\end{proof}

\begin{remark}
When $G/E$ has the form of a semi-direct product $N\rtimes Q$ and $\mathrm{supp}(\tilde{\mu}) \subset U$, we can apply directly Theorem~\ref{thm:norm}. In this case, the constants of Corollary~\ref{cor:EDmax} depend only on that of Theorem~\ref{thm:norm} and the constant of Definition~\ref{def:distortion}.
\end{remark}

\subsection{Solvable groups of finite Pr\"ufer rank}\label{sec:FiniteRank}
Recall that a group has finite Pr\"ufer rank if for some $k$, all its finitely generated subgroups admit a generating $k$-tuple.
Following \cite{CT},  let us abbreviate ``virtually solvable of finite Pr\"ufer rank" to ``VSP".  The proof of Corollary \ref{cor:finireRank} goes in two steps. We first treat the case where $G$ is virtually torsion-free, in which case we embed it as a cocompact lattice in a locally compact group of class $\mathfrak{C}''$, whose definition we recall below.

\begin{definition}\cite{CT}
Denote by $\mathfrak{C}''$ the class of compactly generated locally compact groups $G$ having
 two closed subgroups $U$ and $N$ such that
\begin{enumerate}
\item $U$ is normal and $G=UN$;
\item\label{ssg0} $N$ is a compactly generated locally compact group of polynomial growth;
\item\label{ssg1} $U$ decomposes as a finite direct product $\prod U_i$, where each $U_i$ is normalized by the action of $N$ and is an open subgroup of a unipotent group $\mathbb{U}_i(\mathbb K_i)$ over some non-discrete locally compact field of characteristic zero $\K_i$. 
\item\label{ssg3} $U$ admits a cocompact subgroup $V$ with, for some $k$, a decomposition $V=V_1V_2\ldots V_k$ where each $V_i$ is a subset such that there is an element $t=t_i\in N$ such that $t^{-n}vt^n\to 1$ as $n\to \infty$ for all $v\in V_i$.
\end{enumerate}
\end{definition}

\begin{lemma}\label{lem:C''}
Let $G$ be a group of class $\mathfrak{C}''$, then $U$ is almost ultrametric and $G/U$ has polynomial growth. In other words, $G$ is polynomial-by-AU.
\end{lemma}
\begin{proof}
The second statement directly follows from the fact that $N$ has polynomial growth. We freely employ the notations of \cite[\S 6]{CT}. 
We let $S$ be a compact generating set of $G$.

We choose a norm $\|\cdot\|_{\mathrm{Lie}}$ on the Lie algebra $\mathfrak{u}$ of $U$. Besides, consider a finite-dimensional faithful representation of $U$ as unipotent matrices in  $\text{M}_d(\K)$ and equip the latter with a submultiplicative norm $\|\cdot\|$. 
 We shall use the notation $\preceq$ and $ \simeq$ to mean ``up to multiplicative and additive constants". 
 
 
 Now let $h=h_1\ldots h_n$, where each $h_i\in U$. Since the norm is submultiplicative, we have 
 \[\log \|h\|\leq \log n+ \max_i\log \|h_i\|.\]
By \cite[Lemma 6.12]{CT}, $|h_i|_S\succeq \log \|h_i\|$, hence 
  \begin{equation}\label{eq:log h}\log \|h\|\preceq  \log n+ \max_i|h_i|_S.\end{equation}
We deduce from the second inclusion of \cite[Lemma 6.4]{CT} that $|h|_S\preceq \log \|\log(h)\|_{\mathrm{Lie}}$. Besides, by the line below (6.1) from \cite{CT}, $\log \|h\|\simeq \log \|\log(h)\|_{\mathrm{Lie}}$. Combining these, we get $|h|_S\preceq \log \|h\|$. Together with  (\ref{eq:log h}), this yields
\[|h|_S\preceq  \log n+ \max_i |h_i|_S,\]
and therefore $U$ is almost ultrametric in the sense of Definition \ref{def:distortion}.
\end{proof}

\begin{proof}[Proof of Corollary \ref{cor:finireRank}: virtually torsion-free case.]
By \cite[Theorem 1.16]{CT} (see also the paragraph after \cite[Theorem 1.12]{CT}), every finitely generated, virtually torsion-free VSP group embeds as a uniform lattice in a locally compact group of class $\mathfrak{C}''$. Hence the conclusion follows from Theorem \ref{cor:EDmax} and Lemma \ref{lem:C''}.
\end{proof}

To treat the general case, we need the following  lemma, which is probably well-known to experts. 

\begin{lemma}\label{lem:liftcentred}
Let $p:\tilde{G}\to G$ be a surjective morphism between finitely generated groups. Then for every centred probability measure $\mu$ on $G$, there exists a centred probability measure $\tilde{\mu}$ on $\tilde{G}$ such that $\mu=p_*\tilde{\mu}$. Moreover, if $\mu$ is finitely supported, then one can choose $\tilde{\mu}$ to be finitely supported as well.
\end{lemma}
\begin{proof}
Let $N=\ker p$.
Denote $G_{\mathrm{abtf}}$ the ``torsion-free abelianization of $G$'', i.e.\ the image of $G$ in $G_{\mathrm{ab}}^{\mathbb{R}}=G/[G,G]\otimes \R$. Let $\pi:G\to G_{\mathrm{abtf}}$ and $K=\ker \pi$. We have the following short exact sequence 
\[1\to N/(K\cap N)\to \tilde{G}_{\mathrm{abtf}}\to G_{\mathrm{abtf}}\to 1.\]
Note that $G_{\mathrm{abtf}}$ and $\tilde{G}_{\mathrm{abtf}}$ are finitely generated torsion-free abelian groups, hence so is $N/(K\cap N)$. 

For every $g\in G$, we pick $\tilde{g}\in \tilde{G}$ such that $p(\tilde{g})=g$, and we define $\nu=\sum_{g} \mu(g)\tilde{g}$. Clearly we have $p_*\nu=\mu$ and $\nu$ is finitely supported if $\mu$ is. But $\nu$ might fail to be centred. Note that the surjective morphism $\bar{p}: \tilde{G}_{\mathrm{abtf}}\to G_{\mathrm{abtf}}$ induces a surjective morphism $\bar{p}_\R: \tilde{G}_{\mathrm{ab}}^{\mathbb{R}}\to G_{\mathrm{ab}}^{\mathbb{R}}$. Moreover, we have \[\bar{p}_\R(\E(\tilde{\pi}_*\nu))=\E((\bar{p}\circ\tilde{\pi})_*\nu)=\E((\pi\circ p)_*\nu)=\E(\pi_*\mu)=0,\]
where we have used for the second equality that $\bar{p}\circ\tilde{\pi}=\pi\circ p$.
Hence $\E(\tilde{\pi}_*\nu)\in \ker \bar{p}=N/(K\cap N)\otimes \R$. In plain words, this means that the average value of the projection of $\nu$ on $\tilde{G}_{\mathrm{abtf}}$ lies in the real vector space spanned by the projection of $\ker p$ in $G_{\mathrm{ab}}^{\mathbb{R}}$. 

The strategy now will be to adjust $\nu$ to make it centred using elements of  $\ker p$.
We let $\eta$ be a finitely supported probability measure on $N/(K\cap N)$, such that $\E(\tilde{\pi}_*\eta)=-\E(\tilde{\pi}_*\nu)$: to do this, one can simply lift a probability measure on $N/(K\cap N)$ (which is isomorphic to $\Z^d$ for some $d$) whose expected value coincides with $-\E(\tilde{\pi}_*\nu)$.

Our final probability measure on $\tilde{G}$
is defined as the convolution product $\tilde{\mu}=\nu*\eta$. Since $\eta$ is supported on $N$, we duly have $p_*\tilde{\mu}=p_*\nu=\mu$. On the other hand, since the convolution product commutes with morphisms, we have
\[\tilde{\pi}_*\tilde{\mu}=(\tilde{\pi}_*\nu)*(\tilde{\pi}_*\eta),\]
from which we deduce that 
\[\E(\tilde{\pi}_*\tilde{\mu})=\E(\tilde{\pi}_*\nu)+\E(\tilde{\pi}_*\eta)=0.\]
Hence $\tilde{\mu}$ is centred, and the lemma is proved.
\end{proof}

\begin{proof}[End of the proof of Corollary \ref{cor:finireRank}]
By a result of Kropholler and Lorensen \cite{KroLor}, every finitely generated  group is a quotient of a virtually torsion-free finitely generated VSP group. Hence by Lemma \ref{lem:liftcentred}, it is enough to deal with the case of virtually torsion-free finitely generated VSP groups, which we have already treated.
\end{proof}

\section{Uniform concentration in finite-by-nilpotent groups and controlled splitting}\label{sec:split}

We consider an extension $1\to N\to G\to F\to 1$, where $F$ is a finite group, and $N$ is a simply connected nilpotent Lie group. Denote $\pi:G\to F$ the projection. In this setting, the constants of concentration of Theorem~\ref{thm:wordconcentration} can be chosen depending only on the dimension and degree of nilpotency of $N$ and on the spectral constant of $\mu$, provided the norm is suitably chosen.

\subsection{Controlled splitting and uniform concentration}

It is well known that extensions of finite groups by simply connected nilpotent Lie groups must be split, but we need a quantitative version of this fact.
In what follows, we denote $\text{Conv}(A)$ the convex hull of a subset $A\subset N$, which is defined, identifying $N$ to its Lie algebra, by $\mathrm{Conv}(A)=\exp(\mathrm{Conv}(\log(A)))$.  A group section of $F$ is a subgroup $F'<G$  such that $\pi:F'\to F$ is an isomorphism.

\begin{theorem}\label{thm:QuantSplit}
There exists  a positive integer $k$ only depending on $F$, and on the dimension $d$ and the degree of nilpotency $s$ of $N$, such that for all $S\subset G$ such that  $\langle\pi(S)\rangle=F$, there exists a group section of $F$ contained in $(\widehat{S}\cup\mathrm{Conv}(\widehat{S}^{k}\cap N))^k$, where $\widehat{S}=S\cup S^{-1}\cup \{1_G\}$. 
\end{theorem}

Before proceeding to the proof, let us show how this quantitative splitting theorem permits to obtain uniform concentration.

We denote $\widetilde{S}:=(\widehat{S}\cup\mathrm{Conv}(\widehat{S}^k\cap N))^k$ the subset obtained in Theorem~\ref{thm:QuantSplit}. It contains a section which we simply denote $F$. 
We consider the subset $\pi_N(F\widetilde{S}F)$ of $N$, which is included in $\widetilde{S}^3$ and invariant under the $F$-action. We consider its convex envelope $R:=\mathrm{Conv}(\pi_N(F\widetilde{S}F))$ and  $T=FRF$. 

Even though the definition of $T$ is rather involved, it should be noted that according to Theorem~\ref{thm:QuantSplit} the set $T$ is described in terms of $S^k$ for some exponent $k$ depending only on the dimension and degree of nilpotency of $N$. The following theorem states that the concentration of the $\mu$ random walk on $G$ is uniform for the word norm $T$.

\begin{theorem}\label{thm:uniformconcentration}
Let $1\to N\to G\to F\to 1$ be an exact sequence with $F$ a finite group and $N$ a simply connected nilpotent Lie group. Let $\mu$ be a compactly supported probability measure on $G$. Denote $S=\mathrm{supp}(\mu)$. Let $T$ be the subset of $G$ defined above.
\begin{enumerate}
\item If ${\mu}$ is centred, then $\left(\frac{1}{\sqrt{n}}\max_{k\le n} |w_k|_T \right)_n$ is subgaussian, i.e. there exists constants $c_1,c_2>0$ such that
\[
\sup_{n\in \mathbb{N}}\mathbb{P}\left(\frac{1}{\sqrt{n}} \max_{k\le n} |w_k|_T \ge t \right) \le c_2\exp\left(-c_1t^2\right).
\]
\item In general, if $N$ is $s$-nilpotent, there exists a direction $v_\mu$ in $\mathfrak{n}$ depending only on ${\mu}$ such that 
$\left( \frac{1}{n^{\frac{2s-1}{2s}}} \max_{k\le n}|w_k\exp(-kv_\mu)|_T \right)_n$ is $\frac{2s}{2s-1}$-concentrated, i.e. there exists constants $c_3,c_4>0$ such that
\[
\sup_{n\in \mathbb{N}}\mathbb{P}\left(\frac{1}{n^{\frac{2s-1}{2s}}} \max_{k\le n} |w_k\exp(-kv_\mu)|_T \ge t \right) \le c_4\exp\left(-c_3t^{\frac{2s}{2s-1}}\right).
\]
\end{enumerate}
Moreover, the constants $c_1,c_2,c_3,c_4$ depend only on the Lie group $N$ via its dimension $d$ and its degree of nilpotency $s$, and on the spectral constant $\kappa_\mu$.
\end{theorem}

Before proceeding to the proof, observe that such concentrations inequalities cannot hold if one does not take convex hulls of the generating sets. For instance the probability measure on $\mathbb{R}$ equidistributed on the geometric sequence $\{\pm 2^{-i}\}_{i=1}^\ell$ has a linear speed, with respect to its support, for times below~$\ell$ by comparison with random walks on a hypercube.



\begin{proof}[Proof of Theorem~\ref{thm:uniformconcentration}]
We let $R:=\mathrm{Conv}(\pi_N(F\widetilde{S}F))\subset N$, and consider the convex, compact, and $F$-invariant set
\[
K:=\sum_{i=1}^s [\log R]_i \quad \textrm{where }[\log R]_1=\log R \textrm{ and } [\log R]_i=[\log R, [\log R]_{i-1}].
\]
By Lemma~\ref{lem:convNilp1} applied iteratively, we have $R \subset \exp(K) \subset R^q$ for some $q=q(s) \in \mathbb{N}$ depending only on $s$, which implies
\begin{equation}\label{eq:normsRK}
\forall x \in N, \quad \frac{1}{q}|x|_{R} \le |x|_{R^q} \le |x|_{\exp(K)} \le |x|_R.
\end{equation}
Now let us assume that the subgroup generated by $S$ is not contained in a subgroup of $N\rtimes F$ of  the form $N_1\rtimes F$ for some strict subgroup $N_1$ of $N$. (We may make such an assumption up to replacing $N$ by $N_1$ in the following.) Then $K$ is a symmetric convex compact neighbourhood of zero so it defines a norm on $\mathfrak{n}$, denoted $\|\cdot\|_K$. The bilinearity constant of this norm is bounded:
\[
C_{\|\cdot\|_K}:=\sup\left\{ \|[x,y]\|_K:x,y\in K \right\} \le 2s^4.
\]
Indeed, denote $x=\sum_{i=1}^s x_i$ and $y=\sum_{i=1}^s y_i$ with $x_i,y_i \in [\log T]_i$. 
Then $[x,y]=\sum_{i=1}^s\sum_{j=1}^s[x_i,y_j]$. 
But applying successively Jacobi identities, one has $[[\log T]_i,[\log T]_j] \subset 2s^2[\log T]_{i+j}$ (see the proof of Lemma~\ref{lem:C=1}).

As the Banach-Mazur compactum has bounded diameter, we can find a euclidian norm $\|\cdot\|'_{\mathrm{eucl}}$ at bounded distance of $\|\cdot\|_K$, i.e. for some constant $c_{BM}$ depending only on $\dim(N)$, we have
\[
\forall x \in \mathfrak{n}, \quad \frac{1}{c_{BM}}\|x\|'_{\mathrm{eucl}} \le \|x\|_K \le c_{BM}\|x\|'_{\mathrm{eucl}}.
\]
We turn it into an $F$-invariant norm by setting $\|x\|_{\mathrm{eucl}}:=|F|^{-1}\sum_{f\in F}\|x^f\|'_{\mathrm{eucl}}$. We get
\[
\frac{1}{|F|c_{BM}}\|x\|_K \le \frac{1}{|F|}\|x\|'_{\mathrm{eucl}} \le \|x\|_{\mathrm{eucl}} \le \frac{c_{BM}}{|F|}\sum_{f\in F}\|x^f\|_K=c_{BM}\|x\|_K
\]
It follows that the bilinearity constant of $\|\cdot\|_{\mathrm{eucl}}$ is bounded in terms of $s$, $|F|$ and $c_{BM}$.

We apply Proposition~\ref{prop:normsubadd} to the norm $\|\cdot\|_{\mathrm{eucl}}$. We obtain a norm $\varphi$ with bilinearity constant at most $1$ satisfying $\|x\|_{\mathrm{eucl}}\le \kappa \varphi(x)$ for all $x$ and some $\kappa$ depending only on $s$, $|F|$ and $\dim(N)$. Thus, passing to the associated homogeneous subadditive norm $|\cdot|_{\mathcal{L},\varphi}$, we have
\begin{equation}\label{eq:KinBphi}
K=B_{\|\cdot\|_K}(1) \subset B_\varphi(|F|c_{BM}\kappa)\subset B_{|\cdot|_{\mathcal{L},\varphi}}(|F|c_{BM}\kappa),
\end{equation}
and 
\[
B_{|\cdot|_{\mathcal{L},\varphi}}\left(\frac{1}{c_{BM}}\right)\cap \mathfrak{m}^1=B_{\|\cdot\|_{\mathrm{eucl}}}\left(\frac{1}{c_{BM}}\right)\cap \mathfrak{m}^1 \subset B_{\|\cdot\|_K}(1)=K.
\]
This shows that $K$ satisfies the assumptions of the last point of Proposition~\ref{prop:normsubadd} with $D_2=c_{BM}$ and $C_1=\kappa|F|c_{BM}$. We obtain $C_2$ depending only on $s$, $\dim(N)$, and $|F|$ with
\[
\forall x \in \mathfrak{n}, \quad |\exp(x)|_{\exp(K)} \le C_2(|x|_{\mathcal{L},\varphi}+1).
\]
Using (\ref{eq:normsRK}), we finally get
\begin{equation}\label{eq:normsTL}
\forall g \in G, \quad |g|_T\le |\pi_N(g)|_R+1 \le qC_2\left(|\pi_N(g)|_{\mathfrak{L},\varphi}+1\right).
\end{equation}
The rest of the proof is the same as for Theorem~\ref{thm:wordconcentration}. We use Theorem~\ref{thm:norm} to get concentration in the homogeneous norm $|\cdot|_{\mathfrak{L},\varphi}$. Concentration in word norm follows by using inequality (\ref{eq:normsTL}), in place of inequality (\ref{eq:normsUL}) for the first point and in place of inequality (\ref{eq:UvsL}) for the second. (In the present setting, one can further take directly $g_n=\exp(nv_\mu)$ in the proof of the second point.)

Note finally that as $S=\mathrm{supp}(\mu)$ is included in $K$, inequality (\ref{eq:KinBphi}) implies that $R_\mu \le \kappa|F|c_{BM}$, so that the concentration constants depend only on $d$, $s$ and $\kappa_\mu$.
\end{proof}

There only remains to prove Theorem~\ref{thm:QuantSplit} which is done in the next subsection.

\subsection{Proof of the quantitative splitting theorem.}
The first step consists in proving the theorem in case $N=\R^d$ is abelian and the action of $F$ on $\R^d$ (induced by the action by conjugation of $G$ on itself) is irreducible. In this case, $G$ is isomorphic to the semi-direct product $\R^d \rtimes F$ given by the action. We fix an arbitrary $F$-invariant euclidean norm on $\R^d$. To ease the proof, we further assume that $\pi(S)=F$, which is not restricting (see the end of the proof of Theorem~\ref{thm:QuantSplit})


\begin{proposition}\label{prop:QuantSplit}
Given an irreducible action of $F$ on $\R^d$.
There exists  a positive integer $k$ such that for all $S\subset G$ with  $\pi(S)=F$, there exists a group section of $F$ contained in $(\widehat{S}\cup\mathrm{Conv}(\widehat{S}^{k}\cap \R^d))^k$. 
\end{proposition}

An important observation is that as $F$ has only finitely many irreducible orthogonal representations up to isomorphism, the proposition in fact provides a number $k$ which is valid for all such representations.

To prepare for the proof of the proposition, we introduce two functions on the space of lifts of $F$ : both quantify the failure to be a section. The first one by measuring the distance between the fixed points of elements in the lift and the second how much relations of $F$ are not satisfied.

We let $\Sigma$ be the set of lifts of $F$, whose elements admit fixed points:
\[
\Sigma=\left\{ T \subset G : |T|=|F|, \pi(T)=F, \forall g \in T, \mathrm{Fix}(g)\neq \emptyset\right\},
\]
where $\text{Fix}(g)$ is the set of fixed points of $g$.
We equip $\Sigma$ with the topology of $G^F$ and define two real fonctions on $\Sigma$. Given $T\in \Sigma$, let 
\[
\delta(T)=\min_{x\in\R^d} \sum_{g\in T} d(x,\text{Fix}(g))^2.
\] 
Note that the minimum is indeed attained because $x\mapsto \sum_{g\in T} d(x,\text{Fix}(g))^2$ is a convex non-negative proper function. 

\begin{fact}\label{fact:delta}
If all elements in $T$ have a common fixed point, i.e. $\delta(T)=0$, then $T$ is a section of $F$.
\end{fact}

\begin{proof}
All elements in the subgroup generated by $T$ fix a common point. So no such element is a translation but the identity. Therefore the map $\pi:\langle T \rangle \to F$ is injective, hence an isomorphism and $T=\langle T\rangle$.
\end{proof}

For the second function, we let $\langle F\mid R\rangle$ be the canonical presentation of $F$ where $R$ is the set of relators of the form $f_1f_2f_3$, where $f_1f_2=f_3^{-1}$. Given $w\in R$, we denote $w(T)$ the element of $G$ (actually in $\R^d$) obtained by replacing each letter $f\in F$ by the unique element $g\in T$ such that $\pi(g)=f$.
We then define 
\[
\Delta(T)=\max_{w\in R}\|w(T))\|^2.
\]
Again $\Delta(T)=0$ if and only if $T$ is a section.
Observe that both $\delta$ and $\Delta$ are invariant under conjugation: for all $g\in G$, $\delta(g^{-1}Tg)=\delta(T)$ and similarly for $\Delta$.

Below is our key observation.

\begin{lemma}\label{lem:deltaDelta}
Given an irreducible action of $F$ on $\mathbb{R}^d$,
there exists $c>0$ such that $\Delta(T)\geq c\delta(T)$ for all $T\in \Sigma$.
\end{lemma}
\begin{proof}
 Note that rescaling the norm on $\R^d$ multiplies $\delta$ and $\Delta$ by the same factor. Since we can assume that $\delta(T)>0$, we can therefore assume that it equals $1$. Given $z\in \mathbb{R}^d$, if the minimum in $\delta(T)$ is reached at $x$, then the minimum in $\delta(zTz^{-1)}$ is attained at $x+z$. Hence by conjugation invariance of both $\delta$ and $\Delta$, we may assume that this minimum is reached at $0$. 
Hence it is enough to show that the infimum of $\Delta$ in restriction to  $\Sigma_1:=\{T\in \Sigma\mid \sum _{g\in T}d^2(0,\text{Fix}(g))=1\}$ is positive.
We now claim that $\Sigma_1$ is a compact subset of $G^F$. It is obviously closed. To show that it is compact it remains to show that the translation part of an element $g$, such that $d(0,\text{Fix}(g))\leq 1$ is bounded by $2$. Indeed if $x\in \mathbb{R}^d$ is fixed by $g=yf$, with $y\in \mathbb{R}^d$ and $f \in F$, we have $y+fx=x$.

Now since $\Delta$ is continuous, it reaches its minimum on $\Sigma_1$. But $\Delta(T)=0$ implies by definition that $T$ is a subgroup of $G$. Since $T$ is finite, this would mean that its elements have a common fixed point, and therefore that $\delta(T)=0$.
Hence the minimum of $\Delta$ on $\Sigma_1$ is positive. 
\end{proof}

We shall need the following easy fact.

\begin{lemma}\label{lem:irreducible}
Given an irreducible action of $F$ on $\mathbb{R}^d$,
there exists  $c'>0$ such that  for all unit vector $v\in \R^d$, the closed unit ball $B(0,c')$ is contained in the convex hull of $\{\pm f\cdot  v\mid f\in F\}.$
\end{lemma}
Note again that as there are only finitely many irreducible orthogonal representations of $F$, the constants $c$ and $c'$ in Lemma~\ref{lem:deltaDelta} and~\ref{lem:irreducible} can be chosen to apply to all such representations.

\begin{proof}
Consider the linear morphism $A_v:\R^{F}\to \R^d$ given by $A_v(t)=\sum_{f\in F} t_f f\cdot v$, where $t=(t_f)_{f\in F}$. First of all, by irreducibility, $\{f\cdot v\mid f\in F\}$ forms a generating subset of $\R^d$, meaning that $A_v$ is surjective for all $v\neq 0$. Next, the statement of the lemma is equivalent to the fact that the image under $A_v$ of the unit ball in $\ell^{\infty}(F)$ contains the ball of radius $c'$ in $\R^d$ equipped with its euclidean norm. But since the dimension is finite, it is enough to prove that there exists $c''>0$ such that the image of the unit ball in $\ell^{2}(F)$ contains the ball of radius $c''$ in $\R^d$. For this it is enough to find a left-inverse of $A_v$ of norm at most $1/c''$.

Note that $A_v^*A_v$ is an invertible endomorphism of $\R^d$. Since $v\mapsto A_v^*A_v$ is continuous, the image of the unit sphere is compact in $\GL(\R^d)$, hence there exists $C$ such that the operator norm of $(A_v^*A_v)^{-1}$ is bounded by $C$ for all unit vectors $v$. Since the norm of $A_v^*$ is also bounded by a constant $C'$, it follows that the norm of $(A_v^*A_v)^{-1}A_v^*$ is bounded by $CC'$. 
But $(A_v^*A_v)^{-1}A_v^*$ is a left-inverse of $A_v$, so we are done.
\end{proof}

The final tool needed in the proof of Theorem~\ref{thm:QuantSplit} is the following lemma allowing to extend inductively the generating set by taking convex envelopes and powers.

\begin{lemma}\label{lem:S1S2}
Assume $S_1 \subset \left( \widehat{S}\cup \mathrm{Conv}(\widehat{S}^{k_1}\cap N)\right)^{k_1}$. Then for every $k_2\ge 1$ there exists an integer $k$ depending only on $k_1,k_2, |F|$ and $\mathrm{dim}(N)$ such that
\[
\left( \widehat{S_1}\cup \mathrm{Conv}(\widehat{S_1}^{k_2}\cap N)\right)^{k_2} \subset \left( \widehat{S}\cup \mathrm{Conv}(\widehat{S}^{k}\cap N)\right)^{k}.
\]
\end{lemma}

\begin{proof} As $\widehat{S}$ contains the identity, the sets $ \mathrm{Conv}(\widehat{S}^{k}\cap N)$ increase with $k$, so it is sufficient to prove the inclusions of $\widehat{S_1}$ and of $\mathrm{Conv}(\widehat{S_1}^{k_2}\cap N)$ in a set of the form $\left( \widehat{S}\cup \mathrm{Conv}(\widehat{S}^{\ell}\cap N)\right)^{\ell}$ for some $\ell$, and then we take $k=k_2\ell$. The first inclusion is obvious.

Now an element $g\in \widehat{S_1}^{k_2}\cap N$ is a product of the form $c_1s_1c_2s_2\ldots c_{j}s_j$ with $j\leq k_2$
where $s_i\in \widehat{S}$, and $c_i\in \text{Conv}(\widehat{S}^{k_1}\cap N)$. By moving all $s$-elements to the left, we get $g=s_1s_2\ldots s_j c_1^{s_1}c_2^{s_1s_2}\ldots c_j^{s_1\ldots s_j}$. Since $g\in N$, we have $s_1s_2\ldots s_j \in \widehat{S}^{k_2}\cap N$. 
Note that an automorphism of $N$ is also an automorphism of the Lie algebra (through identification via the exponential map). In particular automorphisms commute with convex combinations. Applying this to conjugation by elements of $G$, we have that each $c_i^{s_1\ldots s_i}$ lies in $\mathrm{Conv}(\widehat{S}^{k_1+2i}\cap N)\subset \mathrm{Conv}(\widehat{S}^{k_1+2k_2}\cap N)$. We deduce that $g\in  \mathrm{Conv}(\widehat{S}^{k_1+2k_2}\cap N)^{2k_2}$. 

We have proved that
$\widehat{S_1}^{k_2}\cap N\subset  \mathrm{Conv}(\widehat{S}^{k_1+2k_2}\cap N)^{2k_2}$, and so $\mathrm{Conv}(\widehat{S_1}^{k_2}\cap N)$ is contained in  $\mathrm{Conv}\left(\mathrm{Conv}(\widehat{S}^{k_1+2k_2}\cap N)^{2k_2}\right)$. 
Lemma \ref{lem:convNilp} applied to $C=\mathrm{Conv}(\widehat{S}^{k_1+2k_2}\cap N)$ and $n=2k_2$ provides an integer $m$ with $\mathrm{Conv}(\widehat{S}_1^{k_2}\cap N)\subset \mathrm{Conv}(\widehat{S}^{k_1+2k_2}\cap N)^m$. The required inclusion holds with $\ell=\max(m,k_1+2k_2)$.
\end{proof}

We now turn to the proof of Theorem \ref{thm:QuantSplit} , starting with the proof of Proposition~\ref{prop:QuantSplit}, the special case  when $N=\R^d$ and $F$ acts irreducibly.

\begin{proof}[Proof of Proposition~\ref{prop:QuantSplit}]
Note that there is no loss of generality in assuming that $S$ is a lift of $F$, namely that $\pi$ induces a bijection $S\to F$. 

The first step consists in reducing to the case where all elements of $S$, when acting by conjugation on $\R^d$, have fixed points. This relies on the fact that each element $g\in\Isom(\R^d)$ has a product decomposition $g=\tau g'$ where $g'$ has a fixed point and $\tau$ is a translation which commutes with $g'$. Assuming that $g\in S$, the fact that $g'$ fixes a point means that the restriction of $\pi$ to the subgroup generated by $g'$ is injective (Fact~\ref{fact:delta}), and therefore that $g'^{|F|}=1$. We deduce that $g^{|F|}=\tau^{|F|}$, and in particular $\tau  \in \mathrm{Conv}(\{g^{\pm |F|}\})$. It follows that 
$g'\in (\{g\}\cup \mathrm{Conv}(\widehat{S}^{|F|}\cap \R^d))^2$. We replace $g$ by $g'$ in $S$. As $\pi(g')=\pi(g)$ we have obtained a set $S'$ with $\pi(S')=F$ and $S' \subset \left( \widehat{S}\cup\mathrm{Conv}(\widehat{S}^{|F|}\cap\mathbb{R}^d)\right)^2$.

Moreover, all elements of $S'$ have fixed points. We shall prove the existence of $k$ such that there exists a group section of $F$ contained in $(S'\cup\mathrm{Conv}(\widehat{S'}^{5}\cap \R^d))^k$, which will establish the proposition.

 Note that conjugating by a translation does not change the conclusion. 
Hence we can assume that the minimum in $\delta(S')$ is reached at $0$. If $\delta(S')=0$, we can take $F'=S'$ by Fact~\ref{fact:delta}. 
Otherwise we may assume that $\delta(S')=1$. This means that each element  $s\in S'$ has a fixed point $x_s$ lying in $B(0,1)$. Hence it is enough to prove that $B(0,1)\subset \mathrm{Conv}(\widehat{S'}^5\cap \R^d)^{k'}$ for some $k'=k'(F)$. Indeed, $s'=x_s sx_s^{-1}$ satisfies that $\pi(s')=\pi(s)$ and $s'$ now fixes $0$. Hence the collection of such $s'$ yields our group $F'$, included in $\left(\widehat{S'}\cup\mathrm{Conv}(\widehat{S'}^5\cap\mathbb{R}^d)\right)^{2k'+1}$. The proposition follows by applying Lemma~\ref{lem:S1S2} with $k_1=\max(|F|,2)$ and $k_2=\max(5,2k'+1)$.

There remains to only prove that $B(0,1)\subset \mathrm{Conv}(\widehat{S'}^5\cap \R^d)^{k'}$. By Lemma \ref{lem:deltaDelta}, we have $\Delta(S')\geq c$, meaning that there exists $w\in R$ such that $\|w(S')\|\geq c$. This gives us an element $v$ of $\R^d\cap S'^3$ of norm at least $c$.
Now by Lemma \ref{lem:irreducible},  we have $B(0,cc')\subset \mathrm{Conv}\{\pm f\cdot  v\mid f\in F\}$. Hence for any integer $k'\geq 1/(cc')$, we have $B(0,1)\subset \mathrm{Conv}(\{\pm f\cdot  v\mid f\in F\})^{k'} \subset \mathrm{Conv}(\widehat{S'}^5\cap \R^d)^{k'}$.  For the last inclusion, note that $f.v$ is the conjugate of $v$ by the element $s'\in S'$ with $\pi(s')=f$.
\end{proof}

\begin{proof}[Proof of Theorem \ref{thm:QuantSplit}] We first assume that $\pi(S)=F$.
The proof is by induction on the dimension of $N$. Let $V$ be an irreducible component of $\mathfrak{n}/[\mathfrak{n},\mathfrak{n}]$ acted upon by $F$. There is a surjective homomorphism $G=N\rtimes F \overset{\sigma}{\longrightarrow} V \rtimes F =:G'$. Then $N_1:=\ker(\sigma)$ is an $F$-invariant subgroup of $N$ with lower dimension.

Applying Proposition~\ref{prop:QuantSplit} to $G'$ provides a section $F'$ of $F$ in $G'$, included in $(\widehat{\sigma(S)}\cup\mathrm{Conv}(\widehat{\sigma(S)}^{k_1}\cap V))^{k_1}$. We lift $F'$ to a subset $S_1$ of $G$ by the following rule : a word in $\widehat{\sigma(S)}$ is simply lifted to the corresponding word in $\widehat{S}$, and a convex combination is lifted to the corresponding convex combination in the Lie algebra of $N$. 
By construction, $S_1$ is contained in $(\widehat{S}\cup\mathrm{Conv}(\widehat{S}^{k_1}\cap N))^{k_1}$ and $\pi(S_1)=F$.

Moroeover, $\pi_V(\langle S_1\rangle\cap N)=F'\cap V=\{0\}$ as $F'$ is a section of $V\rtimes F$. Here $\pi_V:N\to V$ is the natural projection and $\langle S_1\rangle$ is the subgroup of $G$ generated by $S_1$. This shows that the subgroup $G_1=\langle N_1 \cup S_1\rangle$ generated by $N_1$ and $S_1$ is isomorphic to $N_1 \rtimes F$.
Also note that it implies 
\begin{equation}\label{eq:convcapN}
\mathrm{Conv}(\langle\widehat{S_1}\rangle\cap N_1)=\mathrm{Conv}(\langle\widehat{S_1}\rangle\cap N).
\end{equation}

By induction applied to $S_1$ in $N_1\rtimes F=G_1$, we obtain a section $F_2$ in $G_1<G$, with
 \[
 F_2 \subset (\widehat{S_1}\cup\mathrm{Conv}(\widehat{S_1}^{k_2}\cap N_1))^{k_2}=(\widehat{S_1}\cup\mathrm{Conv}(\widehat{S_1}^{k_2}\cap N))^{k_2}
 \]
 The last equality is due to (\ref{eq:convcapN}). Lemma~\ref{lem:S1S2} shows that $F_2$ is included in $(\widehat{S}\cup\mathrm{Conv}(\widehat{S}^{k}\cap N))^k$ for some $k$ depending only on $k_1$, $k_2$ and the degree of nilpotency $s$. The integer $k_1$ given by Proposition~\ref{prop:QuantSplit} depends only on $|F|$ and $\mathrm{dim}(N)$ (as there are only finitely many irreducible orthogonal representations of $F$ in a given dimension). By induction, this is also the case of $k_2$.
 
 There remains to treat the general case with $\langle\pi(S)\rangle=F$. However this assumption guarantees that $S^{|F|}$ projects onto $F$. The general case follows by Lemma~\ref{lem:S1S2}.
 \end{proof}

\appendix

\section{Norms on nilpotent Lie algebras adapted to the lower central series }\label{app:norms}

In this appendix, we recall some results due to Guivar'ch~\cite{Guivarch} about homogeneous norms adapted to the lower central filtration of a nilpotent Lie algebra, and how they relate to word metrics on the corresponding Lie group. We reproduce the arguments to stress the dependance of the constants involved on the dimension and the degree of nilpotency.

We consider a nilpotent Lie algebra $\mathfrak{n}$ and denote $\gamma_i(\mathfrak{N})$ the lower central series defined by $\gamma_1(\mathfrak{n})=\mathfrak{n}$ and $\gamma_{i+1}(\mathfrak{n})=[\mathfrak{n},\gamma_i(\mathfrak{n})]$. We assume $\mathfrak{n}$ is $s$-nilpotent. For each $1 \le i \le s$, we denote by $\mathfrak{m}^i$ the orthogonal complement of $\gamma_{i+1}(\mathfrak{n})$ in $\gamma_i(\mathfrak{n})$. Here we assume that $\mathfrak{n}$ is endowed with a scalar product. Moreover when $\mathfrak{n}$ is acted upon by a compact group $Q$, we may and do assume that this scalar product is $Q$-invariant. We get a decomposition
\begin{align}\label{eq:decomp2}
\mathfrak{n}=\mathfrak{m}^1 \oplus \dots \oplus \mathfrak{n}^s.
\end{align}
We denote $u=u^1+\dots+u^s$ the associated decomposition of a vector $u \in \mathfrak{n}$.
Let $\|\cdot\|_e$ denote the ($Q$-invariant) euclidean norm associated to the scalar product. 
Our aim is to construct a second norm $\varphi$ as in Proposition~\ref{prop:normsubadd}. For this we consider $(\kappa_i)_{i=2}^s$, a sequence of constants to be chosen large enough later. We denote $\varphi_1$ the restriction of $\|\cdot\|_e$ to $\mathfrak{m}^1$, and $B_{\varphi,1}(1)$ its unit ball. We define inductively for $2\le i \le s$
\[
B_{\varphi,i}(1):=\kappa_i[B_{\varphi,1}(1),B_{\varphi,i-1}(1)] \subset \mathfrak{m}^i.
\]
The set $B_{\varphi,i}(1)$ is an open convex symmetric subset of $\mathfrak{m}^i$ so  it is the unit ball of a norm on $\mathfrak{m}^i$ which we denote by $\varphi_i$. We define the norm $\varphi$ on $\mathfrak{n}$ by
\begin{equation}\label{eq:phinorm}
\varphi(u):=\sup_{i \in \{1,\dots,s\}}\varphi_i\left( u^i\right).
\end{equation}

The norm $\varphi$ is related to the original norm $\|\cdot\|_e$ by the following lemma involving the bilinearity constant
\[
C_e:=\sup\{\|[u,v]\|_e: \|u\|_e, \|v\|_e \le 1\}.
\]

\begin{lemma}\label{lem:efi}
Let $B_e(1)$ and  $B_\varphi(1)$ denote the unit balls of the norms $\|\cdot\|_e$ and $\varphi$ respectively. Then
$B_{\varphi}(1) \subset \kappa_s\dots \kappa_2 C_e^{s-1}B_{e}(1)$,
or equivalently for any $v\in \mathfrak{n}$,
\[
\|v\|_e \le \kappa_s\dots\kappa_2C_e^{s-1} \varphi(v).
\]
\end{lemma}

\begin{proof} Let us denote  $B_{e,i}(r):=B_e(r)\cap \mathfrak{m}^i$. By induction we have for $2\le i\le s$
\[
B_{\varphi,i}(1) \subset \kappa_i[B_{e,1}(1),\kappa_{i-1}\dots \kappa_2C_e^{i-2}B_{e,i-1}(1)] \subset \kappa_{i}\dots \kappa_2C_e^{i-1}B_{e,i}(1).
\]
\end{proof}

\begin{lemma}\label{lem:C=1}
If $\kappa_i\ge 2i^2$ for all $2\le i \le s$, then 
\[
C_\varphi:=\sup_{u,v\in B_\varphi(1)} \varphi\left([u,v]\right) \le 1.
\]
\end{lemma}

\begin{proof}
Observe that
\[
\varphi([u,v])=\varphi\left(\sum_{k=1}^s \sum_{i+j=k} [u^i,v^j] \right)=\sup_{k\in\{2,\dots,s\}}\varphi_k\left(\sum_{i+j=k}[u^i,v^j] \right).
\]
By induction it is sufficient to prove that $\varphi_s\left( \sum_{i+j=s} [u^i,v^j]\right) \le 1$ for any $u^i \in B_{\varphi,i}(1)$ and $v^j\in B_{\varphi,j}(1)$, that is 
\begin{align}\label{eq:C=1}
\sum_{i+j=s}[B_i(1),B_j(1)] \subset \kappa_s [B_1(1),B_{s-1}(1)]=:B_s(1).
\end{align}
However by Jacobi identity, we have
\[
[B_i,B_j]=[B_i,[B_1,B_{j-1}]] \subset [B_1,[B_{j-1},B_i]]+[B_{j-1},[B_1,B_i]]\subset [B_1,B_{s-1}]+[B_{j-1},B_{i+1}].
\]
So by an easy induction (\ref{eq:C=1}) holds true as soon as we assume $\kappa_s\ge 2s^2$. 
\end{proof}

Now we define the homogeneous nom adapted to the lower central filtration and the new norm $\varphi$ to be
\begin{align*}
\quad |u|_{\mathfrak{L},\varphi}=\sup_{i\in\{1,\dots,s\}} \varphi(u^i)^{\frac{1}{i}}.
\end{align*}

\begin{lemma}\label{lem:subadd}
There exists explicit universal constants $A_i\ge 1$ for $2\le i \le s$ such that if $\kappa_i\ge 2i^2A_i$, then the homogeneous norm $|\cdot|_{\mathcal{L},\varphi}$ is subadditive :
\[
\forall u,v \in \mathfrak{n}, \quad |u\ast v|_{\mathfrak{L},\varphi} \le |u|_{\mathfrak{L},\varphi}+|v|_{\mathfrak{L},\varphi}
\]
\end{lemma}

\begin{proof} By induction on $s$,
it is enough to show that if $\kappa_s$ is large enough, we have
\begin{align}\label{eq:pt2}
\varphi\left((u\ast v)^s\right)^{\frac{1}{s}} \le |u|_{\mathfrak{L},\varphi}+|v|_{\mathfrak{L},\varphi}
\end{align}
By the Baker-Campbell-Hausdorff formula (see Section~\ref{sec:BCH}), this term has the form
\begin{align}\label{eq:ws}
(u\ast v)^s=u^s+v^s+Q_s(u^1,\dots,u^{s-1},v^1,\dots,v^{s-1})
\end{align}
where $Q_s$ is a finite sum with an explicit number of terms of the form an explicit constant times an iterated bracket $[u^{i_1},[v^{i_2}[\dots[u^{i_{r-1}}],v^{i_r}]]\dots]]$ with $i_1+\dots+i_r=s$ (possibly starting with $v^{i_1}$ or ending with $u^{i_r}$). By Lemma~\ref{lem:C=1} and Fact~\ref{lem:multicomm}, we have
\[
\varphi\left([u^{i_1},[\dots,v^{i_r}]] \right)\le \varphi(u^{i_1})\dots \varphi(v^{i_r}) \le |u^{i_1}|_{\mathfrak{L},\varphi}^{i_1}\dots|v^{i_r}|_{\mathfrak{L},\varphi}^{i_r}\le |u|_{\mathfrak{L},\varphi}^p|v|_{\mathfrak{L},\varphi}^q
\]
for some $p+q=s$ with $p,q\ge 1$.  
 It follows that for some explicit constant $A_s$, depending only on the explicit polynomial $Q_s$, we have
\begin{align*}
\varphi_s\left(Q_s(u^1,\dots,u^{s-1},v^1,\dots,v^{s-1})\right)&\le A_s\sum_{p+q=s}|u|_{\mathfrak{L},\varphi}^p|v|_{\mathfrak{L},\varphi}^q \le A_s\sum_{p=1}^{s-1} {s\choose p}|u|_{\mathfrak{L},\varphi}^p|v|_{\mathfrak{L},\varphi}^q \\  &\le A_s\left((|u|_{\mathfrak{L},\varphi}+|v|_{\mathfrak{L},\varphi})^s-|u|_{\mathfrak{L},\varphi}^s-|v|_{\mathfrak{L},\varphi}^s \right).
\end{align*}
This holds for any choice of $\kappa_s\ge 2s^2$.
In particular this holds for $\varphi_s=\frac{\|\cdot\|'}{2s^2}$ where $\|\cdot\|'$ is the norm with unit ball $[B_1(\varphi),B_{s-1}(\varphi)]$. So~(\ref{eq:ws}) gives
\[
\frac{\|(u\ast v)^s\|'}{2s^2} \le \frac{\|u^s\|'}{2s^2}+\frac{\|v^s\|'}{2s^2}+A_s\left((|u|_{\mathfrak{L},\varphi}+|v|_{\mathfrak{L},\varphi})^s-|u|_{\mathfrak{L},\varphi}^s-|v|_{\mathfrak{L},\varphi}^s \right).
\]
Note that the termin brackets does not depend on $\kappa_s$, so
if we let $\kappa_s\ge2s^2A_s$, we conclude
\[
\varphi\left( (u\ast v)^s\right) \le \varphi(u^s)+\varphi(v^s)+(|u|_{\mathfrak{L},\varphi}+|v|_{\mathfrak{L},\varphi})^s-|u|_{\mathfrak{L},\varphi}^s-|v|_{\mathfrak{L},\varphi}^s \le (|u|_{\mathfrak{L},\varphi}+|v|_{\mathfrak{L},\varphi})^s
\]
which gives~(\ref{eq:pt2}).
\end{proof}

\begin{lemma}\label{lem:homogene-word} Let $\varphi$ be a norm of the form~(\ref{eq:phinorm}) such that Lemma~\ref{lem:C=1} and~\ref{lem:subadd} hold. Let $R$ be a compact generating of $N$ containing a neighbourhood of the identity. There are two constants $C_1,D_2\ge 1$ with
\begin{equation}\label{eq:C1D2}
B_{e,1}\left(\frac{1}{D_2}\right) \subset \log R \subset B_{|\cdot|_{\mathcal{L},\varphi}}(C_1).
\end{equation}
Then
\[
\forall u \in \mathfrak{n}, \quad \frac{1}{C_1} |u|_{\mathfrak{L},\varphi}\le  |\exp(u)|_R \le C_2\left(|u|_{\mathfrak{L},\varphi} +1\right),
\]
for some constant $C_2$ depending only on $C_1$ and $D_2$.
\end{lemma}

Note that for any radius $r$ one has $B_{e,1}(r)=B_{\varphi,1}(r)=B_{|\cdot|_{\mathcal{L},\varphi}}(r)\cap \mathfrak{m}^1$ because $\|\cdot\|_e$, $\varphi$  and $|\cdot|_{\mathcal{L},\varphi}$ coincide on $\mathfrak{m}^1$.
The additive constant on the righthand side of the third point is necessary because the word metric for $R$ takes integer values.

\begin{proof} The existence of $C_1,D_2$ is obvious. We proceed by induction on $s$. For $s=1$, let us recall the classical proof.
Let $g=\exp(u)$ be such that $|g|_R=n$. Then $g=g_1\dots g_n$ with $g_i \in R$. We get $u=\log(g_1)+\dots+\log(g_n) \subset nB_\varphi(C_1)=B_\varphi(nC_1)$ so $|u|_{\mathfrak{L},\varphi} \le C_1 |\exp(u)|_R$. Conversely if we had $|u|_{\mathfrak{L},\varphi} \le \frac{1}{D_2}(n-1)$, we could write $u=(n-1)u'$ with $u' \in B_\varphi(\frac{1}{D_2}) \subset \log(R)$, thus $\exp(u)=\exp(u')^{n-1} \subset R^{n-1}$, which is a contradiction. So $|u|_{\mathfrak{L},\varphi} \ge \frac{1}{D_2}(|\exp(u)|_R-1)$.

Now assume the result holds for $s-1$ nilpotent groups. 
Consider $g=\exp u$ with $|g|_R=n$, then we get $g_1,\dots,g_n \in R$ with $g=g_1\dots g_n$. By Lemma~\ref{lem:subadd}, we have
\begin{align}\label{eq:upC1}
|u|_{\mathfrak{L},\varphi} \le \sum_{i=1}^n |\log g_i|_{\mathfrak{L},\varphi} \le nC_1= C_1|\exp u|_R.
\end{align}

Now consider $u\in \mathfrak{n}$ with $|u|_{\mathfrak{L},\varphi}=\lambda$. Denote $\overline{R}$ the image of $R$ in the quotient $N\to N/\gamma_s(N)$. By induction there exists $K_{s-1}$ explicit such that
\[
|\exp(u^1+\dots+u^{s-1})|_{\overline{R}} \le K_{s-1}(\lambda+1).
\]
For some $K_{s-1}\lambda \le n \le K_{s-1}(\lambda+1)$, one has $\exp(u^1+\dots+u^{s-1})=\overline{g}_1\dots\overline{g}_n$ with $\overline{g}_i \in \overline{R}$. Thus $\exp u=g_1\dots g_nc$ with $c$  in the center $\gamma_s(N)$, which we rewrite as $c=g_n^{-1}\dots g_1^{-1}\exp u$. By Lemma~\ref{lem:subadd}
\[
\varphi_s(\log c)^{\frac{1}{s}}=|\log c|_{\mathfrak{L},\varphi} \le \sum_{i=1}^n |\log g_i^{-1}|_{\mathfrak{L},\varphi}+|u|_{\mathfrak{L},\varphi} \le C_1 n+\lambda \le \left(C_1+\frac{1}{K_{s-1}}\right)n.
\]
Then $\varphi_s(\log c)\le \widetilde{K}_{s}n^s$ with $\widetilde{K}_{s}=\left(C_1+\frac{1}{K_{s-1}}\right)^s$. Thus there exist $\gamma$ in the center $\mathfrak{m}^s=\gamma_s(\mathfrak{n})$ such that $c=\exp(n^s\gamma)$ and $\varphi_s(\gamma) \le \widetilde{K}_{s}$. By definition of the norm $\varphi_s$, this implies that 
\begin{align*}
\gamma \in B_{\varphi,s}(\widetilde{K_s})  =\widetilde{K}_sB_{\varphi,s}(1) & =\widetilde{K}_s\kappa_s\kappa_{s-1}\dots \kappa_2[B_1(1),[\dots,[B_1(1),B_1(1)]]]
\\ 
&=\widetilde{K}_s\kappa_s\dots \kappa_2D_2^s\left[B_1\left(\frac{1}{D_2}\right),\left[\dots,\left[B_1\left(\frac{1}{D_2}\right),B_1\left(\frac{1}{D_2}\right)\right]\right]\right]
\end{align*}
so $\gamma=\widetilde{K}_s\kappa_s\dots \kappa_2D_2^s[y_1,[\dots,[y_{s-1},y_s]]]$ with $y_i \in B_1\left(\frac{1}{D_2}\right)\subset \log(R)$.  We get
\begin{align*}
c&=\exp(\widetilde{K}_s\kappa_s\dots \kappa_2D_2^sn^s[y_1,[\dots,[y_{s-1},y_s]]])\\ &=[\exp(ny_1),[\dots[\exp(ny_{s-1}),\exp(ny_s)]]]^{\widetilde{K}_s\kappa_s\dots \kappa_2D_2^s}
\end{align*}
so that $|c|_R\le \widetilde{K}_s\kappa_s\dots \kappa_2D_2^s4^sn$ and $|\exp u|_R \le n(1+\widetilde{K}_s\kappa_s\dots \kappa_2D_2^s4^s)\le K_s (\lambda+1)$.
\end{proof}

\begin{proof}[Proof of Proposition~\ref{prop:normsubadd}]
The proposition is simply a concatenation of Lemmas~\ref{lem:efi}, \ref{lem:C=1},  \ref{lem:subadd} and \ref{lem:homogene-word}, in the case where $\kappa_i=2i^2A_1$ as in Lemma~\ref{lem:subadd}.
\end{proof}

We finally record two lemmas that will be used in Section~\ref{sec:split}.

\begin{lemma}\label{lem:convNilp1}
Let $N$ be a nilpotent simply connected Lie group of step $s$. There exists a finite sequence of integers $p_1,q_1,p_2,q_2\ldots, p_m,q_m $, only depending on $s$ such that 
for all pair of convex compact subsets $C_1,C_2\subset N$ containing the neutral element, 
\[\exp\left(\log C_1+\log C_2\right)\subset C_1^{p_1}C_2^{q_2}\ldots C_1^{p_m}C_2^{q_m}.\]
and 
\[\exp[\log C_1, \log C_2]\subset C_1^{p_1}C_2^{q_2}\ldots C_1^{p_m}C_2^{q_m}.\]
\end{lemma}
\begin{proof}
This follows from the Lazard formulas (cf \cite[Lemma 5.2]{BreuillardGreen}, according to which  there are sequences of rational numbers
$\alpha_1,\ldots \alpha_m$, $\beta_1,\ldots, \beta_m$, $\gamma_1,\ldots, \gamma_m$, and $\delta_1,\ldots, \delta_m$, only depending on $s$ such that for any $x,y\in N$,
\[\exp(\log x+\log y)=x^{\alpha_1}y^{\beta_1}\ldots x^{\alpha_m}y^{\beta_m},\] 
and
\[\exp([\log x, \log y])=x^{\gamma_1}y^{\delta_1}\ldots x^{\gamma_m}y^{\delta_m}.\] 
Since $C_1$ and $C_2$ are convex, rational powers of $x$ and $y$ are contained in integral powers of $C_1$ and $C_2$, so the lemma is proved.
\end{proof}

We deduce the following useful lemma:

\begin{lemma}\label{lem:convNilp}
Let $N$ be a nilpotent simply connected Lie group of step $s$. For every $n$ there exists an integer $m=m(s,n)$ such that for all convex compact subset $C\subset N$ containing the neutral element, we have
\[\mathrm{Conv}(C^n)\subset C^m.\]
\end{lemma}
\begin{proof}
By a trivial induction on $n$, it is enough to prove it for $n=2$.  By the Baker-Campbell-Hausdorff formula, $C^2$ is contained in the convex set $C':=\exp\left(\sum_{i=1}^s \eta_i [\log C]_i\right)$, with usual notation $[X]_i:=[X,[X]_{i-1}]$ and $[X]_1=X$, and $\eta_i$ explicit rationnal coefficients. It suffices to prove that $C'$ is contained in a bounded power of $C$, which is obtained by successively applying Lemma \ref{lem:convNilp1}. 
\end{proof}


\bibliographystyle{alpha}
\bibliography{References}


\end{document}